\documentclass[12pt]{amsart}
\usepackage{amsmath,amsthm,amssymb,amsfonts,graphicx, hyperref}
\usepackage{tikz}
\usetikzlibrary{shapes.geometric,calc}
\usepackage{color}
\usepackage{todonotes}
\usepackage{bm}
\usepackage[margin = 2.5cm]{geometry}
\usepackage[capitalise]{cleveref}
\usepackage{array}
\usepackage{mdwlist}
\usepackage{algorithmic}
\usepackage{algorithm}
\usepackage{oplotsymbl}
\usepackage[numbers]{natbib}
\usepackage{xparse}
 \usepackage{amsaddr}
\usepackage{xpatch}
\makeatletter
\xpatchcmd{\@settitle}{\uppercasenonmath\@title}{\Large}{}{\PatchFailure}
\makeatother

\title[Low degree sum-of-squares bounds for the stability number]{Low degree sum-of-squares bounds for the stability number: a copositive approach} 
\author{L\MakeLowercase{uis} F\MakeLowercase{elipe} V\MakeLowercase{argas}\textsuperscript{1} \quad \quad
        J\MakeLowercase{uan} C. V\MakeLowercase{era}\textsuperscript{2} 
        \quad \quad P\MakeLowercase{eter} J.C. D\MakeLowercase{ickinson}\textsuperscript{3}}
\thanks{The first author was supported by the Swiss National Science Foundation project No.200021 207429 / 1 Ideal Membership
Problems and the Bit Complexity of Sum of Squares Proofs}

\thanks{\textsuperscript{1}University of Applied Sciences and Arts of Southern Switzerland SUPSI-IDSIA.
E-mail: \href{mailto:luis.vargas@supsi.ch}{luis.vargas@idsia.ch}.}
\thanks{\textsuperscript{2}Department of Econometrics and Operations Research, Tilburg University, The Netherlands. 
E-mail: \href{mailto:j.c.veralizcano@tilburguniversity.edu}{j.c.veralizcano@tilburguniversity.edu}.}
\thanks{\textsuperscript{3}Rabobank, The Netherlands, and University of Twente, The Netherlands. 
E-mail: \href{mailto:peter.jc.dickinson@gmail.com}{peter.jc.dickinson@gmail.com}.}

\newtheorem{defin}{Definition}[section]
\newtheorem{definition}[defin]{Definition}
\newtheorem{proposition}[defin]{Proposition}
\newtheorem{theorem}[defin]{Theorem}

\newtheorem{corollary}[defin]{Corollary}
\newtheorem{lemma}[defin]{Lemma}
\newtheorem{example}[defin]{Example}

\newtheorem{conjecture}{Conjecture}

\DeclareMathOperator{\interior}{int}

\newcommand{\R}{\mathbb{R}}

\newcommand{\COP}{\operatorname{\mathcal{C\!O\!P}}}

\newcommand{\newQ}{\operatorname{\widetilde\CQ}}

\newcommand{\CI}{\operatorname{\mathcal{I}}}
\newcommand{\CN}{\operatorname{\mathcal{N}}}
\newcommand{\CH}{\operatorname{\mathcal{H}}}

\newcommand{\CK}{\mathcal{K}}
\newcommand{\CS}{\mathcal{S}}

\newcommand{\CQ}{\mathcal{Q}}
\newcommand{\CC}{\mathcal{C}}

\newcommand{\qrcert}{\CQ^{(r)}{\rm \text{-certificate}}}
\newcommand{\supp}{\rm\text{supp}}

\newcommand{\N}{\operatorname{\mathbb{N}}}

\newcommand{\nnrank}{\operatorname{\newNu-rank}}

\newcommand{\nrank}{\operatorname{\nu-rank}}

\newcommand{\newNu}{\operatorname{\tilde \nu}}

\newcommand{\vd}{\mathbf{d}}

\newcommand{\val}{{\bm{\alpha}}}

\newcommand{\SOS}[1]{\operatorname{\Sigma}_{#1}}
\newcommand{\Rplus}[1]{\operatorname{\mathcal N}_{#1}}

\usepackage{tikz}
\usetikzlibrary{shapes}
\usetikzlibrary{positioning}
\tikzset{top node/.style={shape=circle,draw,minimum size=0.8cm,inner sep=2pt},}
\tikzset{centre node/.style={star,star points=12,draw,minimum size=0.8cm,inner sep=0pt},}
\tikzset{bottom node/.style={shape=diamond,draw,minimum size=0.8cm,inner sep=1pt},}

\Crefname{subsection}{Subsection}{Subsections}
\crefname{subsection}{Subsection}{Subsections}

\begin{document}

\maketitle

\begin{abstract}

The stability number of a graph \(G\), denoted as \(\alpha(G)\), is the maximum size of an independent (stable) set in \(G\). Semidefinite programming (SDP) methods, which originated from Lov\'asz’s theta number and expanded through lift-and-project hierarchies as well as sums of squares (SOS) relaxations, provide powerful tools for approximating \(\alpha(G)\).

We build upon the copositive formulation of \(\alpha(G)\) and introduce a novel SDP-based hierarchy of inner approximations to the copositive cone \(\COP_n\), which is derived from structured SOS representations. This hierarchy preserves essential structural properties that are missing in existing approaches, offers an SDP feasibility formulation at each level despite its non-convexity, and converges finitely to \(\alpha(G)\).
Our results include examples of graph families that require at least \(\alpha(G) - 1\) levels for related hierarchies, indicating the tightness of the de Klerk–Pasechnik conjecture. Notably, on those graph families, our hierarchy achieves \(\alpha(G)\) in a single step.

\end{abstract}
%\tableofcontents

\section{Introduction}

\nocite{*}

Computing the \emph{stability number} of a graph is a fundamental problem in combinatorial optimization and graph theory. Given a graph $G = (V,E)$, the objective is to find $\alpha(G)$, the size of the largest subset of vertices $S\subseteq V$  so that no two vertices in $S$  are adjacent. Due to the NP-hardness of this problem, optimization hierarchies to approximate or bound the stability number have been proposed in the literature.

Semidefinite programming (SDP) hierarchies, particularly the Lasserre hierarchy for binary optimization~\cite{lasserre2001_0_1} and the Lov\'asz-Schrijver SDP-hierarchy~\cite{lovaszSchrijver} serve as powerful tools for both approximating and computing the stability number of a graph. By systematically introducing high-dimensional semidefinite constraints, these methods achieve tighter relaxations approximating $\alpha(G)$  with increasing accuracy. Despite the computational challenges that still exist, advancements in SDP solvers and approximation techniques position these hierarchies as cornerstones of modern combinatorial optimization. Notably, both hierarchies converge in a finite number of steps when computing \( \alpha(G) \), with the \( \alpha(G)^{th} \)-level being exact in both cases. Furthermore, \citet{LeventExample} construct examples of graph families for which a constant factor of $\alpha(G)$ steps are necessary for the convergence of the Lov\'asz-Schrijver SDP-hierarchy.

This paper explores optimization hierarchies for the stability number grounded in copositive optimization. This cone-based approach effectively captures the combinatorial structure inherent to the stability number problem. Various LP/SDP hierarchies of inner approximations for the copositive cone, \( \COP_n \), have been proposed in the literature (see,  e.g., \cite{deKlerk_Pasechnik_2002,Pena_Vera_Zuluaga_2007, Dickinson_Povh_2013, LV21b}). We study some of these SDP-approximations, which are developed by considering {\em sufficient} conditions for copositivity using sums of squares of polynomials (sos).

By leveraging sos-hierarchies approximating the copositive cone, sequences of SDP-bounds asymptotically converging to $\alpha(G)$ are obtained.
Analyzing the convergence behaviour of these sequences is, in general, very challenging. For instance, finite convergence has been a topic of much interest. \citet{deKlerk_Pasechnik_2002} conjectured that a particular hierarchy converges in $\alpha(G)-1$ steps. Still, only recently, it has been proved that the sequence is finitely convergent~\cite {SV24}. On the other hand,  constructions of graph families for which the hierarchy requires \(\sim\alpha(G) \) steps for convergence remain absent from the literature.

This paper addresses some significant gaps in the existing literature. Firstly, we present families of graphs for which at least \( \alpha(G) - 1 \) steps are necessary for the convergence of a hierarchy closely related to the one used in the Klerk-Pasechnik conjecture. This finding suggests that if the conjecture holds, it may be tight. Secondly, we highlight that the main challenges in analyzing the convergence of such sequences stem from the absence of fundamental structural properties of the corresponding hierarchies. To address this issue, we propose a new hierarchy that approximates the copositive cone, possessing most of these essential properties (see Section~\ref{sec:StructureTilde}). 

However, a limitation of the proposed hierarchy is that its levels are generally not convex. Each level of the hierarchy is described by a \emph{bilinear} SDP-constraint. Nevertheless, the incidence relation for each level can be formulated as a linear matrix inequality corresponding to an SDP feasibility problem. Using this hierarchy, we construct a sequence of upper bounds that converges to \( \alpha(G) \) (see Section~\ref{sec:nuTilde}).

By leveraging the structural properties of the proposed cones, we show the finite convergence of the obtained sequence of upper bounds on $\alpha(G)$. Our bound is equal to $\alpha(G)$ when the level $r= (\tfrac {|G|}{\alpha} + 1)^\alpha$  of the hierarchy is used (see \cref{cor:dGenBound}). Also, we show that for graphs with a constant number of independent sets of size $\alpha(G)$, the number of steps required to converge is independent of $n$. Moreover, we show that our sequence of upper bounds is much stronger than the existing ones (see ). Indeed, we show several classes of graphs for which the level-1 bound equals $\alpha(G)$, while the existing hierarchies require $\sim\alpha(G)$ levels for convergence, which is also linear in the number of nodes.

An interesting byproduct of our construction is that we write the incidence problem for the  $5\times 5$-copositive cone as an SDP-feasibility problem. 
In contrast, \citet{bodirsky2024spectrahedral} has shown that the  $n\times n$-copositive cone is not a spectrahedral shadow for $n \ge 5$.

\subsection{Copositive model of the independence number}

A symmetric matrix \( M \) is copositive if its associated quadratic form \( x^T Mx \) is nonnegative over the nonnegative orthant \( \R^n_+ \). The class of copositive matrices is fundamental within various mathematical domains. Their significance in optimization is highlighted by the ability to model many challenging combinatorial optimization problems as linear optimization over the cone of copositive matrices \cite{bomze2012copositive, dur2010copositive, Burer2015AGG, Dr2021ConicOA}.

Recall the formulation for the independence number given by \citet{deKlerk_Pasechnik_2002}. Given an $n$-vertex graph  $G$, let $A_G$,  $I$, and $J$ denote, respectively, the adjacency matrix of $G$, the identity and the all-ones matrices, all of size $n \times n$. Then, we have the following identity:
\begin{align}\label{alpha-cop}
\alpha(G)=\min \{t : t(A_G +I)-J\in \COP_n\},
\end{align}
where we denote the cone of copositive matrices by
\[
\COP_n = \{M \in \CS^n: x^T Mx \ge 0 \text{ for all }x\in \R^n_+\}.
\]

The computational complexity counterbalances the broad modeling capabilities of the copositive cone. Linear optimization over $\COP_n$ is known to be an NP-hard problem, and the problem of determining whether a matrix is copositive is co-NP-complete~\cite{murty1985coNP}. In light of these computational challenges, some tractable conic approximations for \( \COP_n \) have been introduced, which we will discuss next.

\subsection{Sum-of-squares approximations for $\COP_n$}
A polynomial $p$ is called a {\em sum of squares}  (sos) if $p=\sum_{i=1}^m q_i^2$ for some other polynomials $q_i$. Sums of squares are important because they serve as certificates of nonnegativity; if a polynomial is a sum of squares, it is nonnegative on $\mathbb{R}^n$. We denote the cone of sums of squares on $n$-variables as \( \SOS n \) and the cone of those of degree at most \( d \) as \( \SOS {n,d} \).\\

Now, we recall the following two inner conic approximations for $\COP_n$.
\begin{align}
\CK_n^{(r)}&=\Big\{M\in \CS^n: \Big(\sum_{i=1}^nx_i\Big)^rx^TMx = \sum_{\substack{\beta\in \mathbb{N}^n:|\beta|\leq r+2 \\ |\beta|\equiv r (\text{mod } 2)}} x^{\beta}\sigma_{\beta},\,
  \sigma_{\beta} \in \SOS{n,r+2-|\beta|}\Big\}, \label{cone-K}
\\
\CQ_n^{(r)}&=\Big\{M\in \CS^n: \Big(\sum_{i=1}^nx_i\Big)^rx^TMx= \sum_{\substack{\beta\in \mathbb{N}^n \\ |\beta|=r,r+2}} x^{\beta}\sigma_{\beta},\,
\sigma_{\beta} \in \SOS{n,r+2-|\beta|}\Big\}. \label{eqQr}
\end{align}
The cones $\CK_n^{(r)}$ were introduced by \citet{Parrilo_Thesis} for $r=0$ and $r=1$ and later generalized by \citet{deKlerk_Pasechnik_2002}. They were originally defined as the set of symmetric matrices \(M\) for which the polynomial $(\sum_{i=1}^nx_i^2)^r(\sum_{i,j}M_{ij}x_i^2x_j^2)$ is an sos. Subsequently,  \citet{Pena_Vera_Zuluaga_2007} showed that this definition is equivalent to the one  provided above in (\ref{cone-K}). 

Additionally, the cones $\CQ_n^{(r)}$ were introduced by \citet{Pena_Vera_Zuluaga_2007} as a restrictive version of the cones $\CK_n^{(r)}$, where only sums of squares of degree 0 and 2 are involved in the representation.
These cones capture the interior of the $\COP_n$. That is, we have
\begin{equation}\label{eq:cones}
\text{int}(\COP_n)\subseteq \bigcup \CQ_n^{(r)} \subseteq \bigcup \CK_n^{(r)} \subseteq \COP_n.
\end{equation}

\subsection{Copositive-based sos-approximations to $\alpha(G)$}
By replacing $\COP_n$ with the cones $\CK_n^{(r)}$ (resp. $\CQ_n^{(r)}$) in problem (\ref{alpha-cop}), a sequence of optimization problems approximating $\alpha(G)$ is obtained,
\begin{align}
\label{eq:defTheta} \vartheta^{(r)}(G)= \min \{t\in \R : t(A_G+I)-J\in \CK_n^{(r)}\}, \\
\label{eq:defNu} \nu^{(r)}(G)= \min \{t\in \R : t(A_G+I)-J\in \CQ_n^{(r)}\}.
\end{align}
 These hierarchies provide strong approximations for $\alpha(G)$ even for $r=0$. In fact, we have that $\nu^{(0)}(G)$ coincides with the bound $\vartheta'(G)$; a  strengthening of the well-known Lovász-Theta number, proposed by \citet{S79}. From~\eqref{eq:cones}, the sequences $\nu^{(r)}(G)$ and $\vartheta^{(r)}(G)$ are both monotonically decreasing and converge asymptotically to $\alpha(G)$. Analyzing the convergence behaviour of this sequence of bounds further is very challenging. The primary difficulties arise from the lack of basic structural properties. For instance, the finite convergence of these sequences has garnered significant interest. De Klerk and Pasechnik conjectured that the hierarchy $\vartheta^{(r)}(G)$ converges in $\alpha(G)-1$ steps.
\begin{conjecture}[\citet{deKlerk_Pasechnik_2002}]\label{conj1}
For every graph $G$, 
\[\vartheta^{(\alpha(G)-1)}(G)=\alpha(G).\]
\end{conjecture}

This conjecture has proven very challenging, with only partial results. Recently, \citet{SV24} showed that finite convergence holds. However, no bound on the number of steps required for this convergence has been established. \cref{conj1} has been proven for graphs with $\alpha(G) \leq 8$ \cite{GL07}. Moreover, this result also applies to the hierarchy $\nu^{(r)}(G)$, meaning that $\nu^{(\alpha(G)-1)}(G)=\alpha(G)$ for graphs with $\alpha(G) \leq 8$.
 Whether this result holds for  $\alpha(G)\geq 9$ remains an open question. It even remains open whether the hierarchy $\nu^{(r)}(G)$ has finite convergence.
 
In the following sections,  we limit the scope of our study to the cones $\CQ_n^{(r)}$. While many of the results presented here can be extended to the cones $\CK_n^{(r)}$, for the sake of clarity, we will focus exclusively on the cones $\CQ_n^{(r)}$.

We denote by $\nrank(G)$ the number of steps required for the hierarchy $\nu^{(r)}(G)$ to converge to $\alpha(G)$. 
\begin{align}
\nrank(G)=\min\{ r\in \mathbb{N}: \nu^{(r)}(G)=\alpha(G)\}
\end{align}
Therefore, we have
\[
r \ge \nrank(G) \Longleftrightarrow \alpha(G)(A_G+I) - J \in \CQ_n^{(r)}.
\]

\subsection{Lack of structural properties of $\CQ_n^{(r)}$}
When studying the $\nrank$, isolated vertices play an important role~\cite{GL07,Pena_Vera_Zuluaga_2007}.  For any graph $G$, let $G^\bullet$ be the graph created by adding an isolated vertex to $G$.
We have the following result.
\begin{theorem}[\citet{LV21b}]\label{thm:isolated}
Let $a \ge 0$ be such that  $\nrank(G^\bullet) \leq \nrank(G) +a$ for all graphs~$G$. Then,  $\nrank(G)\leq (a+1)\alpha(G)-1$ for all graphs $G$.
\end{theorem}
Theorem~\ref{thm:isolated} motivates studying how $\nrank$ behaves when adding an isolated node. Initially, it was conjectured by~\citet{GL07} that  $\nrank(G^\bullet) \le \nrank(G)$ for all $G$, which would imply \( \nrank(G) = \alpha(G) - 1 \). However, ~\cite {LV21b} shows this conjecture to be false for the graph obtained by adding eight isolated nodes to the 5-cycle.

 Let a graph $G$ be given. Let $\alpha := \alpha(G)$. We have then $\alpha(G^\bullet) = \alpha+1$.  A first attempt for analyzing $\nrank(G^\bullet)$  is by considering the following identity proposed in \cite{GL07},
\begin{align}
\alpha(G^\bullet)(A_{G^\bullet}+I) - J
\nonumber
&=
(\alpha+1)\begin{pmatrix} 1  & 0 \cr 0 & (I+A_{G})\end{pmatrix} - \begin{pmatrix} 1  & e^T \cr e & J\end{pmatrix}\\
\label{M_G-isolated}
&\qquad=
\begin{pmatrix} \alpha  & -e^T \cr -e & {1\over \alpha}J\end{pmatrix}
+
{\alpha+1\over \alpha}
\begin{pmatrix} 0 & 0 \cr 0 & \alpha(I+A_{G})-J\end{pmatrix}.
\end{align}
The matrix $\begin{pmatrix} \alpha  & -e^T \cr -e & {1\over \alpha}J\end{pmatrix}$ is positive semidefinite and $\alpha(I+A_{G})-J \in  \CQ_n^{(\nrank(G))}$. Thus
the identity~\eqref{M_G-isolated} seems the natural step to obtain bounds on $\nrank(G^\bullet)$.  It is straightforward to show that adding new rows and columns of zeros to a copositive matrix results in another copositive matrix. Therefore, it is natural to anticipate that approximations to the copositive cone are also closed under this bordering operation. However, closure under borderings fails strongly for the cones $\CQ_n^{(r)}$. Specifically, 
\citet{LV21b} show that
$\begin{pmatrix}
M & 0\\
0 & 0
\end{pmatrix} 
 \notin \CQ_{n+1}^{(r)},
$ for any $r\ge 0$ whenever $M\notin \CQ_{n}^{(0)}$, which implies that the second term on the right-hand side of~\eqref{M_G-isolated} does not belong to any cone $\CQ_n^{(r)}$ unless $\nrank(G) = 0$.

Another important property of the copositive cone is closure under diagonal scalings, which involves multiplying by a non-negative diagonal matrix on both the left and right sides.
\citet{Dickinson_Duer_Gijben_Hildebrand_2013} showed that every $5\times 5$ copositive matrix with an all-ones diagonal belongs to the cone $\CQ_5^{(1)}$. This result translates into a copositivity test when $n=5$, as determining whether a matrix $M$ is copositive reduces to checking whether the scaled version of $M$ (with a $0-1$ diagonal) bleongs to $\CQ_5^{(1)}$. Additionally, this implies that if $\CQ_5^{(1)}$ is closed under scalings then $\CQ_5^{(1)}$ equals $\COP_5 $.  
However, $\CQ_5^{(r)}$ is not closed under diagonal scalings for $r > 0$.  To see this, consider the \emph{Horn matrix}
\begin{align}\label{horn}H = {\tiny
{\left(\begin{array}{rrrrr}
 1 & 1 & -1 & -1 & 1\cr
1 & 1 & 1 & -1 & -1\cr
-1 & 1 & 1 & 1 & -1\cr
-1 & -1 & 1 & 1 & 1 \cr
1 & -1 & -1 & 1 & 1
 \end{array}\right)}},
\end{align}
which  belongs to $\CQ_5^{(1)} \setminus \CQ_5^{(0)}$ \cite{Parrilo_Thesis}. \citet{Dickinson_Duer_Gijben_Hildebrand_2013} showed that for any $r>0$ there exists a diagonal scaling of $H$ that do not belong to the cone $\CQ_5^{(r)}$. 
More generally, for $n\geq 5$ and $r\ge 1$, the cone $\CQ_n^{(r)}$ is not closed under scalings \cite{Dickinson_Duer_Gijben_Hildebrand_2013}.

At this point, it is natural to ask whether the copositivity test for $n=5$ based on scaling can used for any larger $n$.  \citet{LV21b} answers this question negatively for $n \ge 7$. They show that whenever $M\notin \CQ_{n}^{(0)}$, then  ${\tiny\begin{pmatrix} 
M & 0 & 0\\
0 & 1 & -1 \\
0 &-1 & 1
\end{pmatrix}}
\notin \CQ_{n+2}^{(r)}$ for any $r$. For instance, using the Horn matrix, when $n\geq 7$,  we can construct copositive matrices of size $n\times n$ with an all-ones diagonal which do not belong to any cone $\CQ_n^{(r)}$. 

To summarize, the ill behaviour of the cones $\CQ_n^{(r)}$ hinders the natural attempts to prove the desired convergence properties of the sequence $\nu^{(r)}(G)$, as basic recursive equations cannot be used due to these cones falling to be closed under bordering and diagonal scalings.

\subsection{Contribution} In this paper, we introduce a hierarchy of cones, denoted as \(\tilde{\CQ}^{(r)}_n\), which is an inner approximation to the copositive cone (see definition \eqref{def:newQ}). These are closed and pointed cones with nonempty interior, closed under borderings, scalings, permutations, and principal submatrices (see \cref{sec:StructureTilde}). It is important to note that, although the cones in this hierarchy are not convex (see \cref{ex:nonConv}), the condition \(M \in \tilde{\CQ}^{(r)}_n\) is a linear matrix inequality (LMI) for any fixed \(r\). These cones satisfy the inclusion relation:
\begin{equation}\label{eq:tildeSubset}
\CQ_n^{(r)} \subseteq \tilde{\CQ}_n^{(r)} \subseteq \COP_n. 
\end{equation}
In particular, the cone \(\tilde{\CQ}_n^{(r)}\) includes any matrix that can be derived from the cone \(\CQ_n^{(r)}\) through diagonal scalings and borderings. In turn, this implies \(\COP_5 = \tilde{\CQ}_5^{(1)}\) (see \cref{COP5-captured}). Consequently, the decision problem \(M \in \COP_5\) is equivalent to an LMI (c.f. \cite{specShadows2024}). 

Considering the cones \(\tilde{\CQ}_n^{(r)}\), we formulate a hierarchy of bounds on the independence number, analogous to \eqref{eq:defNu}: \[ \tilde{\nu}^{(r)}(G) = \min \{t \in \mathbb{R} : t(A_G + I) - J \in \tilde{\CQ}_n^{(r)}\}. \] From the established  inclusion relation \eqref{eq:tildeSubset}, we obtain that \(\alpha(G) \leq \tilde{\nu}^{(r)}(G) \leq \nu^{(r)}(G)\). Thus, it is evident that \(\tilde{\nu}^{(r)}(G)\) converges to \(\alpha(G)\) at least as rapidly as \(\nu^{(r)}(G)\).

Our findings reveal that the hierarchy of bounds $\tilde{\nu}^{(r)}(G)$ is significantly stronger than the hierarchy $\nu^{(r)}(G)$. 
In \cref{sec:nuTilde}, we show finite convergence of the hierarchy  $\tilde{\nu}^{(r)}(G)$; we give bounds on the number of steps necessary for this convergence based on graph structures. In particular, we show that in general upper $r = \left(\tfrac{|G|}{|\alpha(G)} + 1 \right)^{\alpha(G)}$ steps are enough.
Also, in \cref{sec-separation}, we provide multiple classes of graphs for which $\tilde{\nu}^{(1)}(G)=\alpha(G)$. In stark contrast, using the results from \cref{section-examples}, we show that the number of steps required by the hierarchy  $\nu^{(r)}(G)$ to convergence to $\alpha(G)$ is at least $\alpha(G)-2$ for the same graphs. It is important to remark that for these graphs, $\alpha(G)$ is linear on the number of nodes. Some of these graphs are also hard instances for the Lovász-Schrijver hierarchy, also needing a factor of $n$ (i.e. a factor of $\alpha(G)$)  levels to obtain that optimal solution~\cite{LeventExample}.

\section{Preliminaries}
We will abuse the notation and use $\N$ to denote the set of natural numbers including 0. For $\alpha \in \N^n$, we define $|\alpha| := \sum_{i=1}^n \alpha_i$, and $x^\alpha = x_1^{\alpha_1}\ldots x_n^{\alpha_n}$.  We set $\N_r^{n}=\{\alpha\in \N^n: |\alpha|=r\}$. 
Given a polynomial $p = \sum_{\alpha\in \mathbb{\mathbb{N}}_d^n} p_{\alpha}x^\alpha$ of degree at most $d$, we define $\|p\|_1 = \sum_{\alpha\in \N_d^{n}}|p_\alpha|$.
We define the cone
\[
\CN_{n,r}=\Big\{ \sum_{\alpha\in \mathbb{N}_r^n}c_\alpha x^\alpha, c_\alpha\geq 0\Big\}
\]
of homogeneous polynomials of degree $r$ in $n$ variables with nonnegative coefficients.  We also define the following cone of homogeneous polynomials.
\begin{align}
\CH_{n,r} = \Big\{ \sum_{\substack{\beta\in \mathbb{N}^n\\|\beta|=r,r+2} }x^{\beta}\sigma_{\beta},\,
\sigma_{\beta} \in \SOS{n,r+2-|\beta|} \Big \}.
\end{align}
The cones $\CN_{n,r}$ and $\CH_{n,r}$ 
 are proper cones of polynomials non-negative on the non-negative orthant~\cite{Hongbo_13} .
 To simplify the exposition, we introduce the following notation.
\begin{definition}\label{def:orders}
 Let $p,q\in \mathbb{R}[x_1, \dots, x_n]$. We define the following partial orders:
\begin{description}
\item[i)] $p\geq_c q$ if $p-q\in \CN_{n,r}$, for some $r\in \mathbb{N}$.
\item[ii)] $p \sqsupseteq q$ if $p-q\in \CH_{n,r}$, for some $r\in \mathbb{N}$
\end{description}
\end{definition}
The properties stated next follow directly from \cref{def:orders} and will be used in the rest of the paper without further reference.
\begin{lemma}\label{lemma-ineq-basic} Let $p,q \in \R[x_1, \dots, x_n]$.
\begin{description} 
\item [i)] If $p\geq_c q$, then $p\sqsupseteq q$.
\item [ii)] If $p\sqsupseteq q$ and $s\geq_c 0$, then $ps \sqsupseteq qs$
\item [iii)] If $p\sqsupseteq 0$ then , then $p(x) \ge 0$ for all $x \in\mathbb{R}^n_+$.
\end{description}
\end{lemma}

Notice that the definition of the cones $\CQ_{n}^{(r)}$ reads
\begin{align*}
\CQ_n^{(r)}&=\{M\in \CS^n: \Big(\textstyle \sum_{i=1}^nx_i\Big)^r\cdot x^TMx\in \CH_{n,r}\}.
\end{align*}

Now, we introduce a new inner approximation to the copositive cone. Let
\begin{equation}\label{def:newQ}
\tilde{\CQ}_n^{(r)} =\{M\in \CS^n:  p\cdot x^TMx\in \CH_{n,r} \text{ for some } p\in  \Rplus{n,r}, \ \|p\|_1=1\}.
\end{equation}
The cones $\tilde{\CQ}_n^{(r)}$ can be seen as a richer version of the cones $\CQ_n^{(r)}$. Indeed, in the definition of $\tilde{\CQ}_n^{(r)}$ the pre-multiplier of $x^TMx$ is any (scaled) degree-$r$ homogeneous polynomial with nonnegative coefficients, while in the definition of $\CQ_n^{(r)}$ the pre-multiplier of $x^TMx$ is the polynomial $(\sum_{i=1}^nx_i)^r$. By construction, the cones $\CH_{n,r}$ and $\Rplus{n,r}$ contain only polynomials nonnegative on $\R^n_+$.  Hence, 
$$ \CQ_n^{(r)} \subseteq \tilde{\CQ}_n^{(r)}\subseteq \COP_n.$$

Notice that for a given matrix $M$, the query $M\in \tilde{\CQ}_n^{(r)}$ is an LMI. Indeed, the coefficients of the polynomial $p\cdot x^TMx$ are bilinear in the coefficients of $p$ and the entries of $M$, and thus the query $p\cdot x^TMx\in \CH_{n,r}$ is an LMI when $M$ is fixed. Observe that also the condition ``$p\in \Rplus{n,r}$ and $\|p\|_1=1$" is linear in the coefficients of $p$.

\subsection{The 0-rank approximations}
A natural first approach to approximate the cone $\COP_n$  is the cone  consisting of the symmetric matrices that can be written as the sum of a positive semidefinite matrix (psd) and a matrix with nonnegative entries. By definitions~\eqref{cone-K},~\eqref{eqQr}, and~\eqref{def:newQ}  it follows that this cone coincides with the cones $\CK_n^{(0)}$, $\CQ_n^{(0)}$, and  $\tilde{\CQ}_n^{(0)}$, i.e, we have
\[
\CK_n^{(0)}=\CQ_n^{(0)} = \tilde{\CQ}_n^{(0)} =\{M\in \CS^n: M=P+N \text{ for some } P\succeq 0, N\geq 0\}.
\]

Diananda showed that for $n\leq 4$, every copositive matrix can be written as a sum of a psd and a a nonnegative matrices. That is,
$$\COP_n = \tilde{\CQ}_n^{(0)} \text{ for } n\leq 4.$$
This equality does not hold for $n\geq 5$. In fact, the Horn matrix~\eqref{horn}
is copositive and does not belong to the cone $\tilde{\CQ}_5^{(0)}$ \cite{Hall_Newman_1963}.

\subsection*{Sum-of-squares approximations for the stability number}\label{prel-graphs}
The cones \(\tilde{\CQ}_n^{(0)} = \CQ_n^{(0)}\) provide strong bounds when applied to problem (\ref{alpha-cop}). It is known that the bound \(\tilde{\nu}^{(0)}(G) = \nu^{(0)}(G)\) strengthens the Lovász theta number \(\vartheta(G)\), and satisfies the inequalities: \[ \alpha(G) \leq \tilde{\nu}^{(0)}(G) \leq \overline{\chi}(G), \] where \(\overline{\chi}(G)\) represents the clique covering number of \(G\), defined as the minimum number of cliques required to cover all the vertices of \(G\)~\cite{Pena_Vera_Zuluaga_2007}. These inequalities imply that the parameter \(\tilde{\nu}^{(0)}(G)\) is equal to \(\alpha(G)\) when \(\overline{\chi}(G) = \alpha(G)\). In particular, \(\tilde{\nu}^{(0)}(G) = \alpha(G)\) when \(G\) is a perfect graph.

\subsection{Recursive bounds on the $\nrank$}
We now present some properties about the convergence of the hierarchy $\nu^{(r)}(G)$. We first introduce some notation. For a set $S\subseteq V$ of vertices,  the extended neighborhood of $S$ is the set
$$S^\perp = \{ i\in V: i\in S \text{ or } \{i,j\}\in E \text{ for some } j\in S\}.$$
For $i\in V$ we set $i^\perp =\{i\}^\perp$.  We observe that for any stable set $S \subseteq V$, we have
\[
\alpha(G \setminus S^\perp) \le \alpha(G)- |S|,
\]
and the equality hold if and only if $S$ is a stable set of $G$ contained in a stable set of size $\alpha(G)$. For a vertex $i\in V$, we consider the graph
\begin{equation}\label{def:Gi}G_i:= (G\setminus i^\perp) \oplus K_{i^\perp}
\end{equation}
defined as the disjoint union of $G\setminus i^\perp$ and the complete graph with vertex set $i^\perp$. We have $\alpha(G_i) \leq \alpha(G)$, with equality if and only if $i$ belongs to a stable set of size $\alpha(G)$.

For a graph $G$, we define the  matrix
\begin{align*}
    M_G=\alpha(G)(A_G+I)-J.
\end{align*}
In \cite{GL07}, the following inequality was shown:

\begin{equation}\label{equation-1}
 \Big(\sum_{i\in V} x_i\Big) x^T M_Gx
\geq _c \sum_{i\in V} x_i x^T M_{G_i} x.
\end{equation}

The following result is a consequence of~\eqref{equation-1}.
\begin{proposition}\cite{GL07}\label{prop:GL07}
For any graph $G= (V,E)$,
\[
\nrank(G)\leq 1 + \max _{i\in V} \ \nrank(G_i)
\]
\end{proposition}
\cref{prop:GL07} is useful for bounding the $\nrank(G)$ in several cases. For example, it can be used to show that odd cycles and odd wheels have $\nrank$ at most 1. However, an important obstacle to applying this result is the presence of isolated nodes. Specifically, if the graph \(G\) contains an isolated node \(i \in V\), then the graph \(G_i\) is identical to \(G\), and as a result, the inequality provides no meaningful information.

\section{Structural properties of  $\tilde{\CQ}^{(r)}$}\label{sec:StructureTilde}
In this section, we examine the structural properties of the cones $\tilde{\CQ}_n^{(r)}$. We will observe that these cones possess richer structure than $\CQ_n^{(r)}$. Specifically, they are closed under diagonal scalings and borderings. These properties are crucial in \cref{sec:nuTilde} when we analyze the strength of the underlying approximations for the stability number.

The cones $\CH_{n,r}$ are closed under the permutation of variables. In this section, we will show that if, additionally, the matrix $M\in \CS^n$ possesses symmetry properties, then the multiplier $p$ in the definition~\eqref{def:newQ} of $\newQ$ can be assumed to have the same 
symmetry. In particular, if $M \in \newQ^{(r)}$ is highly symmetric, then the corresponding multiplier $p$ has a few different coefficients. For instance, when $r=1$, this implies that $M \in \CQ^{(1)}$. Using this fact, we show that the $\newQ^{(1)}_{10}$ is not convex (see \cref{ex:nonConv}) and obtain the only two graphs for which we know that $M_G \notin \newQ^{(1)}$ (see \cref{ex:Ico}).

\subsection{Bordering and Scalings}
\begin{lemma}\label{lemma-struc-new}
Let $M_0\in \tilde{\CQ}_n^{(r_0)}$, $M_1\in \tilde{\CQ}_n^{(r_1)}$, and $M_2\in \tilde{\CQ}_m^{(r_2)}$. Then,
\begin{itemize}
\item [i)] $\begin{pmatrix} M_1 & 0 \cr 0 & M_2 \end{pmatrix} \in \newQ_{n+m}^{(r_1+r_2)}$
\item [ii)] $M_0 + M_1 \in \newQ_n^{(r_0+r_1)}$.
\end{itemize}
\end{lemma}
\begin{proof}
Let $p_0\in \CN_{n,r_0}$, $p_1\in \CN_{n,r_1}$ and $p_2\in \CN_{m,r_2}$ with $\|p_0\|_1=1, \|p_1\|_1=1, \|p_2\|_1=1$ be such that $p_0x^TM_0x\in \CH_{n,r_0}$, $p_1x^TM_1x\in\CH_{n,r_1}$ and  $p_2x^TM_2x\in \CH_{m,r_2}$. Writing the associated quadratic form for $\begin{pmatrix} M_1 & 0 \cr 0 & M_2 \end{pmatrix}$ as
\(\begin{pmatrix}
x^T & y^T
\end{pmatrix}\begin{pmatrix} M_1 & 0 \cr 0 & M_2 \end{pmatrix} \begin{pmatrix}
     x\\
     y
\end{pmatrix} = x^TM_1x +y^TM_2y,\)

We obtain by Lemma \ref{lemma-ineq-basic},
$$p_1(x)p_2(y)(x^TM_1x +y^TM_2y) = p_2(y)p_1(x)x^TM_1x + p_1(x)p_2(y)y^TM_2y \sqsupseteq 0.$$
Since $0\neq p_1p_2\in \CN_{n+m, r_1+r_2}$, we obtain i). Similarly, we have that
$$p_0p_1x^T(M_0 + M_1)x = p_1p_0x^TM_0x + p_0p_1x^TM_1x \sqsupseteq 0,$$
showing ii), as $0\neq p_0p_1\in \CN_{n, r_1+r_2}$.
\end{proof}
\noindent Notice that by taking $M_2=0$ in \cref{lemma-struc-new}.i) we obtain that the cones $\CQ^{(r)}$ are closed under borderings.

Now we look at scalings. Observe that cones $\CN_{n,r}$ and $\CH_{n,r}$ are closed under scaling the variables by a positive number. The proof follows directly from the definitions.
\begin{lemma}\label{aux-scaling} Let $d_1, d_2, \dots, d_n$ be positive real numbers.
\begin{description}
\item[i)] $p(x_1, \dots, x_n) \in \CN_{n,r} \Longleftrightarrow p(d_1x_1, \dots, d_nx_n) \in \CN_{n,r}$
\item[ii)] $p(x_1, \dots, x_n) \in \CH_{n,r} \Longleftrightarrow p(d_1x_1, \dots, d_nx_n) \in \CH_{n,r}$
\end{description}
\end{lemma}
Using this simple fact, we can show that the cones $\tilde{\CQ}_{n}^{(r)}$ are closed under positive diagonal scalings.
\begin{proposition}\label{diagonal-new}
Let $D=Diag(d_1, d_2, \dots, d_n)$ be a positive diagonal matrix and let $r\in \mathbb{N}$. Then, for every symmetric matrix $M\in \CS^n$, we have
$$ M\in \tilde{\CQ}_n^{(r)} \Longleftrightarrow DMD\in \tilde{\CQ}_n^{(r)} $$
\end{proposition}
\begin{proof}
We only need to show the ``only if" part since the matrix $M$ can be seen as the positive diagonal scaling $D^{-1}(DMD)D^{-1}$ of the matrix $DMD$. Assume $M\in \tilde{\CQ}_n^{(r)}$, then there exists $p\in  \CN_{n,r}$ with $\|p\|_1=1$ such that $p(x_1, \dots, x_n)x^TMx\in \CH_{n,r}$. Now by replacing $x_i\to d_ix_i$, and using Lemma \ref{aux-scaling} ii), we obtain $p(d_1x_1, \dots, d_nx_n)x^TDMDx\in \CH_{n,r}$. This shows that  $DMD\in \tilde{\CQ}_n^{(r)}$ because, by Lemma \ref{aux-scaling} i), we have $0\neq p(d_1x_1, \dots, d_nx_n)\in  \CN_{n,r}$.
\end{proof}

\begin{proposition}\label{newQclosed}
Let $n,r \in \mathbb{N}$, then $\tilde{\CQ}_n^{(r)}$ is closed.
\end{proposition}
\begin{proof}
Let $(M_i)_{i\in\mathbb{N}} \to M$, with $M_i\in \tilde{\CQ}_n^{(r)}$ for all $i\in \mathbb{N}$. Let $p_i \in \CN_{n,r}$ (with $\|p_i\|_1=1$) be such that $p_i\cdot x^TM_ix\in \CH_{n,r}$. Notice that the set $\{p \in  \CN_{n,r}: \|p\|_1 = 1\}$ is a compact set and thus there exists a convergent subsequence $p=\lim_{k\to \infty} p_{i_k}$, with $p \in \CN_{n,r}$ and $\|p\|_1 =1$.
As $\CH_{n,r}$ is closed, we have $p \cdot  x^TMx =  \lim_{k \to \infty} p_{i_k} \cdot x^TM_{i_k}x \in \CH_{n,r}$. Thus $M \in \tilde{\CQ}_n^{(r)}$.
\end{proof}

\begin{corollary}\label{COP5-captured}
$\COP_5=\tilde{\CQ}_5^{(1)}$.
\end{corollary}
\begin{proof}
We know $\tilde{\CQ}_5^{(1)} \subseteq \COP_5$. Also, both cones are closed. Thus it is enough to show $\interior \COP_5 \subset \tilde{\CQ}_5^{(1)}$. Let $M \in \interior \COP_5$. All elements on the diagonal of $M$ are positive. Let $D$ be the diagonal matrix with the same diagonal as $M$. Then $\hat M:= D^{-1/2}MD^{-1/2}$ is a copositive matrix with an all-ones diagonal.
But every such matrix belongs to $\CQ_5^{(1)}$ 
\cite{Dickinson_Duer_Gijben_Hildebrand_2013}.  
Then $\hat M \in \CQ_5^{1} \subset \newQ_5^{(1)}$. But, as $\newQ_5^{(1)}$ is closed under scaling (see \cref{diagonal-new}), we have that $M\in \newQ_5^{(1)}$ too.
\end{proof}

\subsection{Symmetry}

The permutation group $\Pi_n$ act naturally on the set of $n$-variate polynomials by permuting the variables. We define $p^\sigma$, the action of $\sigma \in \Pi_n$ on $p \in \R[x]$, by   $p^\sigma(x_1,\dots,x_n) := p(x_{\sigma(1)},\dots,x_{\sigma_n})$. The group $\Pi_n$ also acts naturally on the set of $n\times n$ symmetric matrices by permuting simultaneously rows and columns. That is, $M^\sigma$ the action of $\sigma \in \Pi_n$ on $M \in \CS^n$ is defined by $M^\sigma_{ij} := M_{\sigma(i)\sigma(j)}$

We observe that if $M$ is invariant under the action of $\sigma$, i.e., $M^\sigma=M$, then the corresponding form $p_M:=x^TMx$ is invariant under the action of $\sigma$, i.e.,
\[
p_M(x_1, \dots, x_n)=p_M(x_{\sigma(1)}, x_{\sigma(2)}, \dots, x_{\sigma(n)}).
\]
We obtain the following result.
\begin{lemma}\label{lemma-sym}
Let $M\in \tilde{\CQ}_n^{(r)}$, and let  $\Pi_M=\{\sigma\in \Pi_n: M^\sigma=M\}$ be the stabilizer subgroup of  $\Pi_n$ with respect to $M$. Then, there exists $\tilde{p}\in \Rplus{n,r+2}$, with $\|\tilde{p}\|_1=1$ such that $\tilde{p}\cdot x^TMx\in  \CH_{n,r}$, and $\tilde{p}$ is invariant under the action of $\Pi_M$, i.e., 
$$\tilde{p}(x_1, \dots, x_n)=\tilde{p}(x_{\sigma(1)}, \dots, x_{\sigma(n)}), \quad \text{ for all }\sigma\in \Pi_M.$$
\end{lemma}
\begin{proof}
    Since $M\in \tilde{\CQ}_n^{(r)}$, there exists $p\in  \Rplus{n,r+2}$, with  $\|p\|_1=1$ such that 
    $$p\cdot x^TMx\in \CH_{n,r}.$$
 Since $\CH_{n,r}$ is a convex cone and it is closed under permuting variables, we have that 

 \[\left(\tfrac{1}{|\Pi_M |}\sum_{\sigma\in S_M} p^\sigma(x)\right) x^TM x =\tfrac{1}{|\Pi_M |}\sum_{\sigma\in S_M} p^\sigma(x)x^TM^\sigma x = \tfrac{1}{|\Pi_M |}\sum_{\sigma\in S_M} (p(x)x^TMx)^\sigma \in \CH_{n,r},
 \]
 and thus taking  $\tilde{p}=\frac{1}{|S_M |}\sum_{\sigma\in S_M} p^\sigma$, the polynomial $\tfrac 1{\|\tilde p\|_1} \tilde p$ satisfies the conditions.
\end{proof}

\begin{corollary}\label{cor:vertex-trans}
Let $G=(V=[n],E)$ be a vertex transitive graph, then $\newNu^{(1)}(G)=\nu^{(1)}(G)$.
\end{corollary}
\begin{proof} 
As $\newNu^{(1)}(G) \leq\ \nu^{(1)}(G)$ always holds, it is enough to show $\newNu^{(1)}(G) \ge \nu^{(1)}(G)$.
Let $t$ be such that $M_t = t(A_G+I) - J \in \tilde{\CQ}^{(1)}_n$. We have that
the stabilizer subgroup $\Pi_{M_t} = Aut_G$, is the automorphism group of $G$. 
From \cref{lemma-sym}, there is $\tilde p(x) = p_1 x_1 + \cdots + p_nx_n \in \CN_{n,1}$ invariant under the action of $Aut_G$, such that $\tilde{p}\cdot x^TMx\in  \CH_{n,1}$. As $G$ is vertex-transitive, for all $i, j \in V$ there is $\sigma \in Aut_G$ such that $\sigma(i) = j$. As $\tilde p$ is invariant under $Aut_G$, we have $p_j = p_i$. Thus $\tilde p  = \frac 1n \sum_i x_i$ which implies $M_t \in \CQ^{(1)}_n$. \end{proof}

In Section \cref{section-examples} we show that for several graphs,
the corresponding matrix $M_G \in \newQ^{(1)}$. Next, we give the only examples we have found where $M_G \notin \newQ^{(1)}$.

\begin{example}\label{ex:Ico}
    Let $G$ be the icosahedron graph in Figure \ref{icosahedron} and let $G^c$ be the complement of this graph. These two graphs are vertex-transitive and have independence number 3. \citet{Dobre_Vera_2015} numerically compute that $\nu^{(0)}(G) \approx 3.2361 \approx \nu^{(1)}(G)$, they also prove that  $\nu^{(2)}(G) = 3 = \alpha(G)$. Similarly, \citet{Bomze2002} show numerically that $\nrank(G^c) > 1$, which implies $\nrank(G^c) = 2$, as \cref{conj1} holds in this case, as $\alpha(G^c) = 3$. As these two graphs are vertex-transitive, by \cref{cor:vertex-trans} we have $\newNu^{(1)}(G) = \nu^{(1)}(G)$ and $\newNu^{(1)}(G^c) = \nu^{(1)}(G^c)$. Thus, the $\nnrank$ of each graph is 2.
\end{example}

\begin{figure}[H]
\begin{tikzpicture}[
    x=.3cm,y=.3cm,
    edge/.style={line width=0.8pt}
  ]
  %% 1) Vertex positions
  \coordinate (v1)  at (-8.02,-11.06);
  \coordinate (v2)  at ( 2.46,  7.18);
  \coordinate (v3)  at (13.02,-11.02);
  \coordinate (v4)  at (-1.00, -3.00);
  \coordinate (v5)  at ( 2.46, -1.00);
  \coordinate (v6)  at ( 5.93, -3.00);
  \coordinate (v7)  at ( 5.93, -7.00);
  \coordinate (v8)  at ( 2.46, -9.00);
  \coordinate (v9)  at (-1.00, -7.00);
  \coordinate (v10) at ( 1.01, -3.99);
  \coordinate (v11) at ( 4.00, -4.00);
  \coordinate (v12) at ( 2.50, -6.59);

  %% 2) Draw all edges
  \foreach \i/\j in {
    1/2,2/3,3/1,
    4/5,5/6,6/7,7/8,8/9,9/4,
    5/2,4/2,6/2,
    4/1,1/9,1/8,8/3,7/3,6/3,
    10/11,11/12,12/10,10/5,5/11,
    10/4,10/9,12/9,12/8,7/12,7/11,11/6}
    {
    \draw[edge] (v\i) -- (v\j);
  }

  %% 3) Draw each vertex as an empty circle
  \foreach \i in {1,...,12}{
    \draw[fill=black, draw=black] (v\i) circle (3.5pt);
    %\node[white,font=\footnotesize] at (v\i) {\i};
  }
\end{tikzpicture}
\caption{Icosahedron graph}\label{icosahedron}
\end{figure}
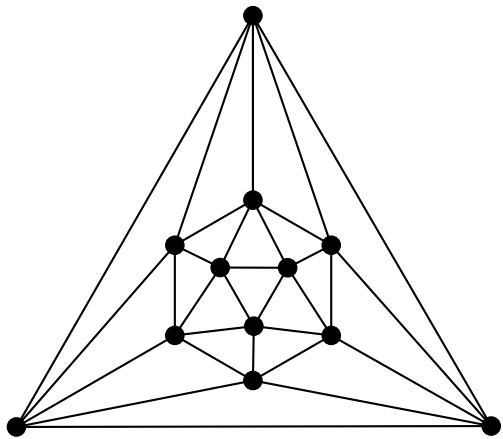

\begin{example}[No convexity]\label{ex:nonConv}
    Let $H \in \CQ_{5}^{(1)} \subset \tilde \CQ_{5}^{(1)}$ be the Horn matrix defined in (\ref{horn}). By Lemma \ref{lemma-struc-new} the matrices $\tiny \begin{pmatrix} H & 0 \cr 0 & 0 \end{pmatrix}$ and $\tiny\begin{pmatrix} 0 & 0 \cr 0 & H \end{pmatrix}$ belong to $\tilde{Q}_{10}^{(1)}$. Now we show that the matrix $H'=\begin{pmatrix} H & 0 \cr 0 & H \end{pmatrix}$ does not belong to  $\tilde{Q}_{10}^{(1)}$, which implies that $\tilde{Q}_{10}^{(1)}$ is not convex. We assume by the sake of contradiction that $H'\in \tilde{Q}_{10}^{(1)}$. Observe that $H'$ is invariant under the action the permutations $\sigma_1=(1,2,3,4,5)$, $\sigma_2=(6,7,8,9,10)$ and $\sigma_3=(1,6)(2,7)(3,8)(4,9)(5,10)$. Then, by Lemma \ref{lemma-sym}, we have that 
    $$\left(\textstyle \sum_{i=1}^{10}x_i \right)x^TH'x\in \CH_{n,r},$$
    and thus $H'\in \CQ_{10}^{(1)}$, which is a contradiction since $H'\notin \bigcup \CQ_{10}^{(r)}$ \cite{LV21b}.
\end{example}

\section{The hierarchy $\newNu^{(r)}(G)$}\label{sec:nuTilde}
Recall that the hierarchy is defined as
\begin{align*}
\tilde{\nu}^{(r)}(G)= \min \{t\in \R : t(A_G+I)-J\in \tilde{\CQ}_n^{(r)}\}.
\end{align*}

 The condition $t(A_G+I)-J\in \newQ_n^{(r)}$ is bilinear in the decision variable  $t$  (which is the objective) and in the coefficients of the polynomial $p$ used in the definition~\eqref{def:newQ} of $\newQ_n^{(r)}$. However, for fixed $t$, the query $t(A_G+I)-J\in \tilde{\CQ}_n^{(r)}$ is an LMI (and thus an SDP feasibility problem), and the answer to such query is `YES' if and only if $t \ge \tilde{\nu}^{(r)}(G)$. Then, we can approximate $\tilde{\nu}^{(r)}(G)$ via binary search in the interval $[1,n]$.

We now define the $\nnrank$ of a graph, as the analogous of the $\nrank$, i.e., the number of steps needed for the hierarchy $\tilde{\nu}^{(r)}(G)$ to converge to $\alpha(G)$.
\begin{align}
\nnrank(G)=\min\{ r\in \mathbb{N}: \tilde{\nu}^{(r)}(G)=\alpha(G)\},
\end{align}
where we define $\nnrank(G) = 0$ for the empty graph $G$.

The following lemma shows a relation between  the $\tilde{\nu}\text{-rank}$ of two graphs $G$, $H$ and its disjoint union $G \oplus H$. This shows, in particular, that adding isolated nodes to a graph does not increase its $\nnrank$.
\begin{lemma}\label{lemma-sum}
Let $G=([n],E(G))$ and $H=([m],E(H))$. Then,
$$\nnrank(G\oplus H)\leq \nnrank(G) + \nnrank(H).$$
In particular, if $H$ is perfect (e.g. a complete graph) $\nnrank(G\oplus H) \le \nnrank(G)$.
\end{lemma}
\begin{proof}
Assume $\nnrank(G)=r_1$ and  $\nnrank(H)=r_2$, so that $M_G\in \tilde{\mathcal{Q}}_{n}^{(r_1)}$ and $M_H\in \tilde{\mathcal{Q}}_{m}^{(r_2)}$.
Observe that the following identity holds
\begin{align}
\label{eq:GplusH}
M_{G\oplus H}
&= \begin{pmatrix} M_G + \alpha(H)(I_n + A_G) & -J_{n,m} \cr -J_{m,n} & M_H + \alpha(G)(A_H + I_m)\end{pmatrix}\\ \nonumber
&= \begin{pmatrix} M_G\left(1+\tfrac{\alpha(H)}{\alpha(G)}\right) + \tfrac{\alpha(H)}{\alpha(G)}J_{n,n} & -J_{n,m} \cr -J_{m,n} & M_H\left(1+\frac{\alpha(G)}{\alpha(H)}\right) + \tfrac{\alpha(G)}{\alpha(H)}J_{m,m} \end{pmatrix} \\
&= \left(1+\tfrac{\alpha(H)}{\alpha(G)}\right) \begin{pmatrix} M_G& 0 \cr 0 & 0 \end{pmatrix} +\left(1+\tfrac{\alpha(G)}{\alpha(H)}\right)\begin{pmatrix} 0 & 0 \cr 0 & M_H\end{pmatrix} + \begin{pmatrix} \frac{\alpha(H)}{\alpha(G)}J_{n,n} & -J_{n,m} \cr -J_{m,n} & \frac{\alpha(G)}{\alpha(H)}J_{m,m}\end{pmatrix}.  \nonumber
\end{align}
We look at the last expression. The first matrix belongs to $\newQ_{n+m}^{(r_1)}$ because $M_G\in \newQ_n^{(r_1)}$ and the cones $\newQ^{(r)}$ are closed under bordering. Similarly, the second matrix belongs to $\newQ_{n+m}^{(r_2)}$. The right-most matrix is positive semidefinite and therefore belongs to $\newQ_{n+m}^{(0)}$. Finally, using Lemma \ref{lemma-struc-new} ii), we obtain the desired result.
\end{proof}

\subsection{Convergence of the hierarchy $\newNu^{(r)}(G)$}
We now analize the convergence of the hierarchy $\newNu^{(r)}(G)$ to $\alpha(G)$. Our main results consist of 1) giving conditions that ensure that $\nnrank(G)\leq 1$ and 2) showing the finite convergence of the hierarchy for all graphs, i.e.,  $\nnrank(G)$ is finite for every graph $G$.  Moreover, we will give explicit bounds on $\nnrank(G)$.

First, for $i\in V$, let $G_i$ the graph defined in~\eqref{def:Gi}. We consider the matrices 
\begin{equation}\label{P(i)}
P_i:=M_{G_i} + (\alpha(G)-\alpha(G_i)) (I+A_{G_i}) = \alpha(G)(I+A_{G_i}) -J.
\end{equation}
Let $\vd=(d_1, d_2, \dots, d_n)\in \mathbb{R}_+^n$, then the following identity holds.

\begin{equation}\label{equation-d}
 \Big(\sum_{i\in V} d_ix_i\Big) x^T M_Gx
= \sum_{i\in V} d_ix_i x^T P_i x
+\alpha(G)\sum_{i\in V} d_ix_ix^T (A_G- A_{G_i}) x\Big.
\end{equation}

In the following lemma, we characterize the vectors $\vd$ for which the second summand on the right-hand side of this identity has all its coefficients nonnegative.
\begin{lemma}\label{lemma-ineq-d}
Let $G=(V=[n], E)$ be a graph and let $\vd=(d_1, \dots, d_n)\in \R_+^n$. For $i\in V$, let $P_i$ be the matrix defined in (\ref{P(i)}). Then, the polynomial $\sum_{i\in V} d_ix_i x^T (A_G-A_{G_i}) x$
has nonnegative coefficients if and only if
\begin{align}\label{ineq-d}
d_j+d_k\geq d_i \quad \text{ for all distinct } i,j,k \in V \text{ such that } \{i,j\}, \{i,k\} \in E, \{j,k\}\notin E.
\end{align}
\end{lemma}
\begin{proof}

 First, observe that, for $i\in V$, the coefficient of the monomial $x_i^3$ is $d_i(A_G-A_{G_i})_{i,i} = 0$. Now, for two distinct vertices $i,j\in V$, the coefficient of $x_i^2x_j$ equals 
\[2d_i(A_G-A_{G_i})_{i,j} + d_j(A_G-A_{G_j})_{i,i} = 2d_i(A_G-A_{G_i})_{i,j} = 0,\]
where the equality holds as $(A_G)_{i,j} = (A_{G_i})_{i,j}$. Finally, we consider the monomials of the form $x_ix_jx_k$, for three distinct vertices $i,j,k\in V$, which equals 
\[2d_i(A_G-A_{G_i})_{j,k} + 2d_j(A_G-A_{G_j})_{i,k} + 2d_k(A_G-A_{G_k})_{i,j}.\]
Observe that the following holds
$$
(A_G-A_{G_i})_{j,k} =
     \begin{cases}
       -1 &\quad\text{if } \{j,k\}\notin E \text{ and } \{i,j\}, \{i,k\}\in E  \\
        1  & \quad \text{if } \{j,k\}\in E \text{ and } (\{i,j\}\in E, \{i,k\}\notin E \text{ or } \{i,j\}\notin E, \{i,k\}\in E)\\
       0 & \quad \text{otherwise}
     \end{cases}
$$

\begin{figure}[H]
\begin{tikzpicture}[
    x=.3cm,y=.3cm,
    edge/.style={line width=0.8pt}
  ]
  %% 1) Vertex positions
  \coordinate (vj)  at (5,0);
  \coordinate (vi)  at (10,5);
  \coordinate (vk)  at (15,0);

  %% 2) Draw all edges
  \foreach \i/\j in {
    i/j,i/k} {
    \draw[edge] (v\i) -- (v\j);
  }

  %% 3) Draw each vertex 
  \foreach \i in {i,j,k}{
    \draw[fill=black, draw=black] (v\i) circle (6.5pt);
    \node[white,font=\footnotesize] at (v\i) {$\i$};
  }
\end{tikzpicture}
\caption{Graph $L_3$}\label{fig-t}
\end{figure}
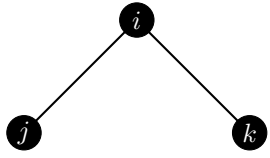

Therefore, the coefficient of $x_ix_jx_k$ is zero unless the subgraph induced by $\{i,j,k\}$ takes the form (in some order) of the graph $L_3$ in Figure \ref{fig-t}. 
In this case, the coefficient of $x_ix_jx_k$ is $-2d_i+2d_j+2d_k$, and thus the coefficients of $\sum_{i\in V} d_ix_i x^T (A_G-A_{G_i}) x$ are nonnegative if and only if condition (\ref{ineq-d}) holds.
\end{proof}

Since $P_i - M_{G_i} \ge 0$ for all $i\in V$, we obtain the following result by combining identity (\ref{equation-d}) and the result of Lemma~\ref{lemma-ineq-d}.

\begin{lemma}\label{lemma-M_G_i}
Let $\vd \in \mathbb{R}_+^V$ be a vector satisfying the inequalities (\ref{ineq-d}). Then, the following inequality holds.
\begin{equation}\label{inequalities-weight}
 \Big(\sum_{i\in V} d_ix_i\Big) x^T M_Gx
\geq _c \sum_{i\in V} d_ix_i x^T M_{G_i} x
\end{equation}
\end{lemma}
This result is a generalization of inequality (\ref{equation-1}). Indeed, we recover~\eqref{equation-1} by taking as $\vd$ the all-ones vector, which clearly satisfies  (\ref{ineq-d}).

Now, we can present a recursive bound on $\nnrank(G)$.

\begin{proposition}\label{prop:bound-rank1}
Let $G$ be a graph, and let $0\neq \vd=(d_1, d_2, \dots, d_n)\in \R^n_+$ be a vector satisfying the inequalities (\ref{ineq-d}). Then
\begin{align}
\nnrank(G)\leq 1+ \sum_{i\in V: d_i>0} \nnrank(G\setminus i^\perp)
\end{align}
\end{proposition}
\begin{proof}
Let $r_i:=\nnrank(G\setminus i^\perp)$. By Lemma~\ref{lemma-sum}, $\nnrank((G\setminus i^\perp)\oplus K_{i^\perp}) =  \nnrank(G\setminus i^\perp)=r_i$. Then, $M_{G_i}\in \newQ_n^{(r_i)}$ for every $i\in V$. Let $q_i \in \Rplus{n,r_i}$ with $\|q_i\|_1=1$ be such that
\[
q_i \cdot x^TM_{G_i}x \sqsupseteq 0.
\]
Observe that inequality (\ref{inequalities-weight}) holds because $\vd$ satisfies (\ref{ineq-d}). Now, we multiply by $\prod_{j\in V: d_j>0}q_j$ at both sides of the inequality (\ref{inequalities-weight}) to obtain, 
\[
 \Big(\prod_{j\in V: d_j>0}q_j\Big)\Big(\sum_{i\in V} d_ix_i\Big) x^T M_Gx  \ge_c \sum_{i\in V} \Big(\prod_{\substack{j\in V: \\ j \neq i, d_j>0}}q_j\Big) d_i x_i q_i\cdot x^T M_{G_i} x  \sqsupseteq 0.
\]

Thus, $\nnrank(G) \le 1 +  \deg  \Big(\prod_{j\in V: d_j>0}q_j\Big)  = 1+ \sum_{i: d_i>0}r_i$.
\end{proof}

\subsection*{Graphs with low $\nnrank$} 
We have seen in Section \ref{prel-graphs} that a sufficient condition for a graph $G$ to have $\nnrank(G)=0$ is that $\overline{\chi}(G)=\alpha(G)$. This condition is satisfied, in particular, by all perfect graphs. We now use proposition~\ref{prop:bound-rank1} to give a sufficient condition that implies that $\nnrank$ of a graph is at most 1.

\begin{proposition}\label{prop-lrank1}
Let $G=(V,E)$ be a graph. Assume there exists $S\subseteq V$ satisfying the following two properties,
\begin{itemize}
 \item [(i)] For every $i\in S$, $\nnrank(G\setminus i^\perp)=0$.
 \item [(ii)] For every two distinct vertices $i,j\notin S$ such that $\{i,j\}\notin E$, there is no $k\in S$ such that $\{i,k\}\in E$ and $\{j,k\}\in E$.
\end{itemize}
 Then, $\nnrank(G)\leq 1$.
\end{proposition}
\begin{proof}
The result follows by applying Proposition ~\ref{prop:bound-rank1} taking $\vd= \chi^S \in \mathbb{R}^V$ as the indicator vector of $S$, i.e., $d_i=1$ if $i\in S$ and $d_i=0$ if $i\in V\setminus S$.
\end{proof}
This result will be used in Section \ref{sec-separation} to show that several graphs $G$ with high $\nrank(G)$ have $\nnrank(G)=1$. As an illustration, we show the following example.
\begin{example}
We consider the graphs $B$ and $C$ shown in \cref{fig.graphB} and \cref{fig.graphC}. These graphs have been shown to be the smallest graphs with $\nrank$ 2 \cite{LV21b}. We show that their  $\nnrank$ equals 1, thus showing that the bounds $\tilde{\nu}^{(r)}(G)$ are stronger than $\nu^{(r)}(G)$. Since $\nu^{(0)}(B)=\tilde{\nu}^{(0)}(B)>\alpha(B)$, and $\nu^{(0)}(C)=\tilde{\nu}^{(0)}(C)>\alpha(C)$, we have that $\nnrank(B)\geq 1$ and $\nnrank(C)\geq 1$. Here, we show that they have $\nnrank$~1.  We apply \cref{prop-lrank1} by taking (in each graph) $S$ as the vertices marked with $\rhombusfill$. Condition (i) holds because for each $i\in S$, the graph $B\setminus i^\perp$ (resp. $C\setminus i^\perp$) is perfect as it does not have cycles, and thus $\nnrank(B\setminus i^\perp)=0$ (resp. $\nnrank(C\setminus i^\perp)=0$). It is straightforward to check Condition (ii). Then, $\nnrank(B)\leq1$ and $\nnrank(C)\leq 1$.
\end{example}

\begin{figure}[H]
\begin{minipage}{0.4\linewidth}
\begin{center}
\definecolor{uuuuuu}{rgb}{0.27,0.27,0.27}
\definecolor{uququq}{rgb}{0.25,0.25,0.25}
\begin{tikzpicture}[line cap=round,line join=round,=triangle 45,x=.65cm,y=.65cm]

\draw [line width=1pt] (-4.43,4.86)-- (-3.55,6.09);
\draw [line width=1pt] (-3.55,6.09)-- (-2.12,5.64);
\draw [line width=1pt] (-2.12,5.64)-- (-0.69,6.11);
\draw [line width=1pt] (-0.69,6.11)-- (0.21,4.91);
\draw [line width=1pt] (-0.67,3.68)-- (0.21,4.91);
\draw [line width=1pt] (-0.67,3.68)-- (-2.10,4.13);
\draw [line width=1pt] (-3.53,3.65)-- (-2.10,4.13);
\draw [line width=1pt] (-3.53,3.65)-- (-4.43,4.86);
\draw [line width=1pt] (-2.10,4.13)-- (-2.12,5.64);
\draw [line width=1pt] (-3.55,6.09)-- (-0.69,6.11);

\node[shape=diamond, fill=uququq, minimum size=11pt, inner sep=0pt] at (-3.53,3.65) {};

\node[shape=diamond, fill=uququq, minimum size=11pt, inner sep=0pt] at (-3.55,6.09) {};

\node[shape=diamond, fill=uququq, minimum size=11pt, inner sep=0pt] at (-2.12,5.64)  {};

\node[shape=diamond, fill=uququq, minimum size=11pt, inner sep=0pt] at (-2.10,4.13) {};

\node[shape=diamond, fill=uququq, minimum size=11pt, inner sep=0pt] at (-0.69,6.11)  {};

\node[shape=diamond, fill=uququq, minimum size=11pt, inner sep=0pt] at (-0.67,3.68) {};

\begin{scriptsize}

\draw [fill=uuuuuu] (-4.43,4.86) circle (2.5pt);
\draw[color=uuuuuu] (-4.39,5.26) node {};

\draw [fill=uuuuuu] (0.21,4.91) circle (2.5pt);
\draw[color=uuuuuu] (0.24,5.30) node {};

\end{scriptsize}
\end{tikzpicture}
\caption{Graph $B$}\label{fig.graphB}
\end{center}
\end{minipage}
\begin{minipage}{0.4\linewidth}
\begin{center}
\definecolor{uuuuuu}{rgb}{0.27,0.27,0.27}
\definecolor{uququq}{rgb}{0.25,0.25,0.25}
\begin{tikzpicture}[line cap=round,line join=round,=triangle 45,x=1cm,y=1.1cm]

\draw [line width=1pt] (-4.65,4.06)-- (-3.18,4.06);
\draw [line width=1pt] (-3.18,4.06)-- (-1.70,4.06);
\draw [line width=1pt] (-1.70,4.06)-- (-0.23,4.06);
\draw [line width=1pt] (-0.23,4.06)-- (-0.23,2.58);
\draw [line width=1pt] (-0.23,2.58)-- (-1.70,2.58);
\draw [line width=1pt] (-1.70,2.58)-- (-3.18,2.58);
\draw [line width=1pt] (-3.18,2.58)-- (-4.65,2.58);
\draw [line width=1pt] (-4.65,2.58)-- (-4.65,4.06);
\draw [line width=1pt] (-3.18,2.58)-- (-1.70,4.06);
\draw [line width=1pt] (-1.70,2.58)-- (-3.18,4.06);

\node[shape=diamond, fill=uququq, minimum size=11pt, inner sep=0pt] at (-3.18,2.58) {};

\node[shape=diamond, fill=uququq, minimum size=11pt, inner sep=0pt] at (-3.18,4.06) {};

\node[shape=diamond, fill=uququq, minimum size=11pt, inner sep=0pt] at (-1.70,4.06) {};

\node[shape=diamond, fill=uququq, minimum size=11pt, inner sep=0pt] at (-1.70,2.58) {};

\begin{scriptsize}
\draw [fill=uuuuuu] (-4.65,4.06) circle (2.5pt);
\draw[color=uuuuuu] (-4.51,4.30) node {};
\draw [fill=uuuuuu] (-4.65,2.58) circle (2.5pt);
\draw[color=uuuuuu] (-4.51,2.85) node {};
\draw [fill=uuuuuu] (-0.23,2.58) circle (2.5pt);
\draw[color=uuuuuu] (-0.09,2.85) node {};
\draw [fill=uuuuuu] (-0.23,4.06) circle (2.5pt);
\draw[color=uuuuuu] (-0.09,4.30) node {};
\end{scriptsize}
\end{tikzpicture}
\caption{Graph $C$}\label{fig.graphC}
\label{G_8}
\end{center}
\end{minipage}

\end{figure}
\subsection*{Finite convergence of the hierarchy $\tilde{\nu}^{(r)}(G)$}

We are ready to show that $\newNu^{(r)}(G)$ converges finitely for all graphs $G$.
\begin{proposition}
Let $G$ be a graph, then $\tilde{\nu}^{(r)}(G)=\alpha(G)$ for some $r\in \mathbb{N}$.
\end{proposition}
\begin{proof} 
We assume by contradiction that the graph $G=([n],E)$ is a minimal counterexample. By the definition of $\newNu$-rank, $G$ can not be empty. We are assuming that $\nnrank(G)=\infty$, and that for every graph $H$ with $|V(H)|<n$ we have $\nnrank(H)<\infty$. Then, for all $i\in V$, we have $\nnrank(G\setminus i^\perp)<\infty$. Then, by Proposition~\ref{prop:bound-rank1} taking $\vd$ as the all-ones vector, we obtain $\nnrank(G)<\infty$, reaching a contradiction. \end{proof}

Now that we have proven that the $\nnrank$ is finite, we proceed to obtain bounds on this rank for several classes of graphs. 
\subsection*{General bounds on $\nnrank(G)$}

Given a graph $G$, a vector $0\neq \vd=(d_1, d_2, \dots, d_n)\in \R^n_+$ and $S \subset V$ let
\begin{equation}\label{def:m_S}
m^G_{S,\vd} = \sum_{i \in V \setminus S^\perp} d_i x_i, \text{  with  } m_{S,\vd}^G = 1 \text{ when } V = S^\perp.
\end{equation}
When $G$ is clear, we simply write $m_{S,\vd}$ instead of $m^G_{S,\vd}$.
Given a graph $G$ we define $$\CI(G) = \{S \subset V: S \text{ independent set in }G\},$$ where $\emptyset$ is considered an independent set, that is $\emptyset \in \CI(G)$ for all $G$. Also, we set
\[\CI^-(G) =\{S \in \CI(G): |S| < \alpha(G)\} \text{ and } \CI^{\max}(G) = \{S \in \CI(G): |S| = \alpha(G)\}.\]  
Next, we construct a multiplier that shows finiteness of  $\nnrank(G)$ for all $G$.
\begin{theorem}\label{thm:multiplier}
Let $G$ be a graph, and let $0\neq \vd=(d_1, d_2, \dots, d_n)\in \R^n_+$ be such that the inequalities (\ref{ineq-d}) hold, and for all $S \in \CI(G)$, such that $\supp(\vd) \subseteq S^\perp$ we have  $M_{G \setminus S^\perp} \in \newQ^{(0)}$. Then,
\[ \left(\prod_{S \in \CI^-(G[\supp\ \vd])} m^G_{S,\vd}\right) x^T M_G x \sqsupseteq 0.\]
\end{theorem}
\begin{proof}
Using $m_{\emptyset,\vd} = \sum_{i\in V} d_ix_i$, by Lemma \ref{lemma-M_G_i} we have

\begin{equation}\label{eq:finite1}
m_{\emptyset,\vd}\ x^T M_Gx = \Big(\sum_{i\in V} d_ix_i\Big) x^T M_Gx \ge_c \sum_{i\in V} d_ix_i x^T M_{G\setminus i^\perp \oplus K_{i^\perp}} x.
\end{equation}
By~\eqref{eq:GplusH}, for any $i\in V$, we have,
\begin{equation}\label{eq:finite2}
x^TM_{G\setminus i^\perp \oplus K_{i^\perp}}x \sqsupseteq \left(1+\tfrac{1}{\alpha(G \setminus i^\perp)}\right) x^T\begin{pmatrix} M_{G\setminus i^\perp}& 0 \cr 0 & 0 \end{pmatrix}x.
\end{equation}
Combining~\eqref{eq:finite1} and~\eqref{eq:finite2} we obtain,
 \begin{equation}
\label{eq:finite3}  m_{\emptyset,\vd} x^T M_G x  \sqsupseteq  \sum_{i\in V} \left(1+\tfrac{1}{\alpha(G \setminus i^\perp)}\right) d_ix_i x^T\begin{pmatrix} M_{G\setminus i^\perp}& 0 \cr 0 & 0 \end{pmatrix}x.
\end{equation}

Let $D = \supp\ \vd$. We prove the statement of the theorem by induction on $\alpha(G[D])$. Assume, $\alpha(G[D]) = 1$. Let $i \in D$. We have $D\subseteq {i}^\perp$, and thus, by the hypothesis of the problem, $G \setminus i^\perp \in \newQ^{(0)}$. This implies $\begin{pmatrix} M_{G\setminus i^\perp}& 0 \cr 0 & 0 \end{pmatrix} \in \newQ^{(0)}$,
as $\newQ^{(r)}$ is closed under bordering for all $r$. By~\eqref{eq:finite3} we have then $m_{\emptyset,\vd} x^T M_G x  \sqsupseteq 0$, which proves the case $\alpha(G[D]) = 1$, as $\CI^-(G[D]) = \{\emptyset\}$.

Let $k \ge 1$. Assume now that the theorems' statement is true whenever $\alpha(G[D]) \le k$ and that $\alpha(G[D]) = k+1$. By~\eqref{eq:finite3} we obtain,
\begin{align}
\label{eq.finite4}  \left(\prod_{S \in \CI^-(G[D])} m_{S,\vd}\right) x^T M_G x  \hspace{-3cm}
& \hspace{3cm} \sqsupseteq \left(\prod_{\substack{S \in \CI^-(G[D]):\\S \neq \emptyset}} m_{S,\vd}\right) \sum_{i\in V} \left(1+\tfrac{1}{\alpha(G \setminus i^\perp)}\right) d_ix_i x^T\begin{pmatrix} M_{G\setminus i^\perp}& 0 \cr 0 & 0 \end{pmatrix}x\\
\nonumber &= \sum_{i\in V} \left(1+\tfrac{1}{\alpha(G \setminus i^\perp)}\right) d_ix_i \left( \prod_{\substack{S \in \CI^-(G[D]):\\ i \notin S, S \neq \emptyset}} m_{S,\vd}\right)  \left( \prod_{\substack{S \in \CI^-(G[D]):\\ i \in S}} m_{S,\vd}\right)  x^T \begin{pmatrix} M_{G\setminus i^\perp}& 0 \cr 0 & 0 \end{pmatrix}x.
\end{align}

And thus, to finish the proof, it is enough to show for each $i \in D$, we have 
\begin{equation}\label{eq.finite5}
 \left( \prod_{\substack{S \in \CI^-(G[D]): i \in S}} m_{S,\vd}\right)  x^T \begin{pmatrix} M_{G\setminus i^\perp}& 0 \cr 0 & 0 \end{pmatrix}x \sqsupseteq 0.
\end{equation}
Fix $i \in D$. Let $\hat G = G  \setminus i^\perp $, and $\hat \vd$ (resp. $\hat x$) be obtained from $\vd$ (resp. $x$) by dropping the entries from $i^\perp$. We have that $\{S \in \CI^-(G[D]): i \in S\} = \{S \cup \{i\}: S \in \CI^-(G[\hat D])\}$, where $\hat D := \supp\ \hat \vd$ $= D \setminus i^\perp$. Also, for all $S \in \CI(G[\hat D])$, we have $m^G_{S \cup \{i\},\vd} (x) = m^{\hat G}_{S,\hat \vd} (\hat x)$. Therefore,
\begin{align*}
\left( \prod_{\substack{S \in \CI^-(G[D]): i \in S}} m_{S,\vd}\right)  x^T \begin{pmatrix} M_{G\setminus i^\perp}& 0 \cr 0 & 0 \end{pmatrix}x
& = \left( \prod_{S \in \CI^-(G[\hat D])} m^{G}_{S\cup \{i\},\vd}\right) x^T\begin{pmatrix} M_{\hat G}& 0 \cr 0 & 0 \end{pmatrix}x \\
& = \left( \prod_{S \in \CI^-(\hat G[\hat D])} m^{\hat G}_{S,\hat \vd}\right) \hat x^T M_{\hat G} \hat x.
\end{align*}
Notice that $\hat \vd$ satisfies the inequalities in (\ref{ineq-d}) in $\hat G$. Also, for any $S \in \CI(\hat G)$, such that $\hat D \subseteq S^\perp$ we have  $D \subseteq (S \cup \{i\})^\perp$. Also, $\hat G \setminus S^\perp = G \setminus (S \cup \{i\})^\perp$. Thus, $M_{\hat G \setminus S^\perp } = M_{G \setminus (S \cup \{i\})^\perp} \in \newQ^{(0)}$.
As $\alpha(\hat{G}[\hat D]) =\alpha(G[\hat{D}])<\alpha(G[D])$, by the induction hypothesis, we have
\[
\left( \prod_{S \in \CI^-(\hat G[\hat D])} m^{\hat G}_{S,\hat \vd}\right) \hat x^T M_{\hat G} \hat x \sqsupseteq 0,
\]
proving~\eqref{eq.finite5}.
\end{proof}
From Theorem~\ref{thm:multiplier} we obtain a bound in the $\nnrank$.

\begin{corollary}\label{cor:dGenBound} Let $G$ be a graph.
Then  $\nnrank(G) \le  \left(\tfrac {|G|}{\alpha(G)}+1 \right)^{\alpha(G)}$
\end{corollary}
\begin{proof} Let $n = |G|$ and $\alpha = \alpha(G)$.
Let $d_i = 1$ for all $i \in G$, which clearly satisfies the inequalities in (\ref{ineq-d}). By \cref{thm:multiplier} we obtain that $\nnrank(G) \le |\CI^-(G)|$.
For all integers $0 \le k \le \alpha$, the number of independent sets in $G$ of size $k$ is at most ${\alpha \choose k} \left( \frac n{\alpha}\right)^k$~\cite{zykov1952}.
We have then, 
\[ \nnrank(G) \le \sum_{k=1}^{\alpha-1}  {\alpha \choose k} \left( \frac {n}{\alpha}\right)^k     
\le  \left(\tfrac {n}{\alpha} + 1\right)^{\alpha} \]
\end{proof}

For the class of graphs with a constant number of independent sets of maximal size, we obtain a bound on $\nnrank(G)$ independent of the size of the graph.

\begin{proposition}\label{prop:Imaxbound}
  Let $G$ be a graph.  Then $\nnrank(G)\le \alpha(G)^2 + 2^{\alpha(G)}|\CI^{\max}(G)|$.
\end{proposition}
\begin{proof}  Let $\alpha = \alpha(G)$. Given $S \subseteq V$ define $m_S := \sum_{i\in V\setminus S^\perp} x_i$. Also, given $T = \{t_1,\dots,t_k\} \in \CI(G)$ consider the graph $G_T := (\cdots((G_{t_1})_{t_2})\cdots)_{t_k}$, and define  $\CI_T^{\max} := \{U\in \CI^{\max}(G): T \subset U\}$. 

We will prove by induction the following stronger statement: \begin{multline}\label{eq:AuxStrong}
\text{ Let $0 \le k \le \alpha$ and $T = \{t_1,\dots,t_{k}\} \in \CI(G)$ be given. We have,}\\
\left( m_{\emptyset}^{\alpha^2} \prod_{U \in \CI^{\max}_T}\prod_{T \subseteq S \subseteq U}  m_S \right)   x^T(\alpha (I + A_{G_T})  -J ) x \sqsupseteq 0.
\end{multline}
Notice that taking $k=0$ we obtain $\nnrank(G)\le \alpha(G)^2 + 2^{\alpha}|\CI^{\max}(G)|$.

To prove~\eqref{eq:AuxStrong} we proceed by induction on $\alpha - k$. when $k= \alpha$, $G_{t_1,\cdots,t_k}$  is the disjoint union of $\alpha$ cliques, which has $\tilde{\nu}$-rank 0.
Now, let $k \ge 0$ and assume the statement holds for all $k'$ such that $k < k'\le \alpha$.
If $\CI_T^{\max} = \emptyset$, then 
$\alpha(G_T) < \alpha$ and thus, 
\[\left( m_{\emptyset}^{\alpha^2} \prod_{U \in \CI^{\max}_T}\prod_{T \subseteq S \subseteq U}  m_S \right)  x^T (\alpha (I + A_{G_T}) + J) x =  m_{\emptyset}^{\alpha^2}  x^T (\alpha (I + A_{G_T}) + J)
x \geq_c 0, 
\]
where the last inequality follows from~\cite[Theorem 1]{Pena_Vera_Zuluaga_2014}.

We assume now $\CI_T^{\max} \neq \emptyset$. Let $d$ be the indicator of $V \setminus T^\perp$, that is $d_i = 1$ if $i \in V \setminus T^\perp$ and $d_i =0$ if $i \in T^\perp$. We have that $d$ satisfies~\eqref{ineq-d} in the graph $G_T$. We use the following equation 
similar to~\eqref{equation-d} \begin{align}
\nonumber m_{T} x^T(\alpha (I + A_{G_T}) + J)x
& = \Big(\sum_{i\in V} d_ix_i\Big) x^T(\alpha (I + A_{G_T}) + J)x \\
\nonumber&= \sum_{i\in V} d_ix_i x^T( \alpha (I + A_{(G_T)_i}) + J) x +\alpha \sum_{i\in V} d_ix_ix^T (A_{G_T} - A_{(G_T)_i}) x \\
\label{eq:Aux1Imaxbound}&\sqsupseteq \sum_{i\in V\setminus T^\perp} x_i x^T( \alpha (I + A_{(G_T)_i}) + J) x.\qquad \text{(by \cref{lemma-ineq-d})}
\end{align}
For each $i \in V \setminus  T^{\perp}$ we have $ T \cup \{i\} \in \CI(G)$. 
Therefore,
\begin{align*}
 &   \left( m_{\emptyset}^{\alpha^2} \prod_{U \in \CI^{\max}_T}\prod_{T \subseteq S \subseteq U}  m_S \right)  x^T(\alpha (I + A_{G_T})  -J ) x\\
  &  =\left( m_{\emptyset}^{\alpha^2} \prod_{U \in \CI^{\max}_T}\prod_{T \subsetneq S \subseteq U}  m_S \right)  m_T^{|\CI^{\max}_T|} x^T(\alpha (I + A_{G_T})  -J )\\
  &  \sqsupseteq
     \left( m_{\emptyset}^{\alpha^2} \prod_{U \in \CI^{\max}_T}\prod_{T \subsetneq S \subseteq U}  m_S \right)  m_T^{|\CI^{\max}_T|-1} \sum_{i\in V\setminus T^\perp} x_i x^T( \alpha (I + A_{(G_T)_i}) + J) x,
     \qquad\text{by equation~\eqref{eq:Aux1Imaxbound}}
     \\
&=
m_T^{|\CI^{\max}_T|-1} \sum_{i\in V\setminus T^\perp} x_i      \left( \prod_{\substack{U \in \CI^{\max}_T\\i\notin U}}\prod_{T \subsetneq S \subseteq U}  m_S \right)\left( m_{\emptyset}^{\alpha^2} 
\prod_{U \in \CI^{\max}_{T\cup \{i\}}}
\prod_{T\cup \{i\} \subset S \subseteq U}  m_S 
\right)  x^T( \alpha (I + A_{(G_T)_i}) + J) x.
\end{align*}
And then~\eqref{eq:AuxStrong} follows, as by induction hypothesis for each $i \in V\setminus T$,
\[\left( m_{\emptyset}^{\alpha^2} 
\prod_{U \in \CI^{\max}_{T\cup \{i\}}}
\prod_{T\cup \{i\} \subset S \subseteq U}  m_S 
\right)  x^T( \alpha (I + A_{(G_T)_i}) + J) x \sqsupseteq 0
\]
\end{proof}
\section{Graphs with high $\nrank$}\label{section-examples}

In this section, we analyze the bounds $\nu^{(r)}(G)$ introduced by Peña, Vera and Zuluaga \cite{Pena_Vera_Zuluaga_2007}. In particular, we develop a technique that allows us to lower bound the $\nrank$ of certain graphs. The results of this section will be used in Section \ref{sec-separation} to exhibit several families of graphs for which the $\nnrank$ and $\nrank$ differ, thus showing that the bounds $\tilde{\nu}^{(r)}(\cdot)$ are stronger than the bounds $\nu^{(r)}(\cdot)$. The results of this section were included in the PhD thesis~\cite{LFV-thesis}.

Our technique builds on the work by Laurent and Vargas \cite{LV21b},  where the authors introduced the notion of $\CK^{(0)}$ and $\CK^{(1)}$-certificates. We recall the notion of  critical edges. An edge $e$ in a graph $G$ is {\em critical} if $\alpha(G\setminus e)=\alpha(G)+1$. For example, in an odd cycle $C_{2n+1}$ ($n\geq 1$) all edges are critical, and no edge is critical in an even cycle $C_{2n}$ (for $n\geq 2$). Equivalently, and edge $e=\{i,j\}$ is critical in $G$ if there exist a set $S\subseteq V$ such that $S\cup\{i\}$ and $S\cup\{j\}$ are stable of size $\alpha(G)$. For example, in the 5-cycle $C_5$ the edge $\{3,4\}$ is critical, as the sets $\{1,3\}$ and $\{1,4\}$ are stable of size $\alpha(C_5)=2$.

 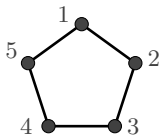
\begin{figure}[H]
  \centering
\definecolor{uuuuuu}{rgb}{0.26666666666666666,0.26666666666666666,0.26666666666666666}
\definecolor{uququq}{rgb}{0.25098039215686274,0.25098039215686274,0.25098039215686274}
\begin{tikzpicture}[line cap=round,line join=round,=triangle 45,x=.3cm,y=.3cm]
%\clip(-14.726363636362857,4.301074380165286) rectangle (-7.156115702477376,11.044876033057852);
\draw [line width=1pt] (-13.578509963461387,8.326106157908189)-- (-11.19999999999867,10.054194799648329);
\draw [line width=1pt] (-11.19999999999867,10.054194799648329)-- (-8.821490036535954,8.326106157908187);
\draw [line width=1pt] (-8.821490036535954,8.326106157908187)-- (-9.72999999999844,5.53);
\draw [line width=1pt] (-12.6699999999989,5.53)-- (-9.72999999999844,5.53);
\draw [line width=1pt] (-12.6699999999989,5.53)-- (-13.578509963461387,8.326106157908189);
\begin{scriptsize}
\draw [fill=uququq] (-12.6699999999989,5.53) circle (2.5pt);
\draw[color=uququq] (-13.66, 5.52) node {$4$};
\draw [fill=uququq] (-9.72999999999844,5.53) circle (2.5pt);
\draw[color=uququq] (-8.92, 5.52) node {$3$};
\draw [fill=uuuuuu] (-8.821490036535954,8.326106157908187) circle (2.5pt);
\draw[color=uququq] (-8.0, 8.52) node {$2$};
\draw [fill=uuuuuu] (-11.19999999999867,10.054194799648329) circle (2.5pt);
\draw[color=uququq] (-12, 10.5) node {$1$};
\draw [fill=uuuuuu] (-13.578509963461387,8.326106157908189) circle (2.5pt);
\draw[color=uququq] (-14.3, 8.9) node {$5$};
\end{scriptsize}
\end{tikzpicture}
 \caption{Graph $C_5$}
 \label{FigC51}
  \end{figure}

The following characterization of the zeros of the form $x^TM_Gx$ on the standard simplex will be very useful. There, critical edges play a crucial role.
\begin{theorem}[\citet{LV21a}]\label{minimizers-LV}
Let $x\in \Delta_n$ with support $S:=\{i\in V:x_i>0\}$ and let $V_1, V_2, \dots, V_k$ denote the connected components of  $G[S]$, the subgraph of $G$ induced by the support $S$ of $x$. Then, $x^TM_Gx=0$ if and only if $k=\alpha(G)$ and, for all $h\in [k]$, $V_h$ is a clique of $G$ and $\sum_{i\in V_h}x_i=\frac{1}{\alpha(G)}$. In addition, the edges that are contained in $S$ 
 are critical edges of $G$.
\end{theorem}

\subsection{Certifying membership in the cones $\CQ^{(r)}$}\label{sect-5.1}
We recall the definition of the cones $\CQ_n^{(r)}$ in relation (\ref{eqQr}): A symmetric matrix $M\in \CS^n$ belongs to $\CQ_n^{(r)}$ if
\begin{align*}
\Big(\sum_{i=1}^{n}x_i\Big)^{r} x^TMx =\sum_{\substack{\beta\in \mathbb{N}^n, \\ |\beta|= r,r+2} }\sigma_\beta x^\beta
\end{align*}
for some  $\sigma_\beta\in \Sigma_{r+2-|\beta|}$.
Observe that we can assume that, in the decomposition, the monomials $x^\beta$ with $|\beta|=r+2$ are square-free. Otherwise, it can be moved to a term of the form $\sigma_\beta x^\beta$ with $|\beta|=r$. The sums of squares  $\sigma_\beta$ corresponding to the monomials $x^\beta$ with $|\beta|=r$ have degree 2, and thus take the form $x^T P_{\beta}x$ for some $n\times n$ positive semidefinite matrix $P_\beta$.

In summary, $M\in \CQ_n^{(r)}$ if, for every $\beta\in \mathbb{N}^n$ with $|\beta|=r$, there exist positive semidefinite matrices $P_\beta$, and, for every $A\subseteq [n]$ with $|A|=r+2$, there exist nonnegative scalars $c_A$ such that
\begin{align}\label{qr-equation}
\Big(\sum_{i=1}^{n}x_i\Big)^{r} x^TMx = \sum_{\beta\in \mathbb{N}^n_r} x^\beta x^TP_\beta x + \sum_{\substack{A\subseteq [n]\\ |A|=r+2}} c_Ax^A,
\end{align}
 where $x^A := \prod_{i\in A}x_i$. We say that $P =  (P_{\beta})_{\beta \in \mathbb{N}_{r}^n}$ is a {\em $\qrcert$} for $M$ if there exist some scalars $c_A\geq 0$ for $A\subseteq [n]$ with $|A|=r+2$ for which Equation (\ref{qr-equation}) holds. The notion of $\qrcert$ is a generalization of the notions of $\CK^{(0)}$ and $\CK^{(1)}$-certificates (recall that $\CK^{(0)}=\CQ^{(0)}$ and $\CK^{(1)}=\CQ^{(1)}$) introduced by Laurent and Vargas in~\cite{LV21b}.

 We now show a result about the structure of the kernel of the matrices in a $\qrcert$.
 \begin{lemma}\label{lemma-ncert} Let $P =  (P_{\beta})_{\beta \in \mathbb{N}_{r}^n}$ be a $\qrcert$ for $M$ and let $a\in \mathbb{R}_+^n$ be such that $a^TMa=0$. Then, for all $\beta \in \mathbb{N}_{r}^n$ such that  $\supp(\beta)\subseteq \supp (a)$ we have $P_\beta a = 0$.
\end{lemma}
\begin{proof}
Let $c_A$ (for $A\subseteq [n]$ with $|A|=r+2$) be nonnegative scalars such that relation (\ref{qr-equation}) holds. By evaluating Equation (\ref{qr-equation}) at $a$, the left-hand side equals zero, and all terms on the right-hand side are nonnegative. Hence, every term on the right-hand side should vanish. In particular, if $\supp(\beta)\subseteq \supp$$(a)$, then $a^\beta > 0$. This implies that $a^T P_\beta a=0$. Hence, $P_\beta a=0$ as $P_\beta\succeq 0$. 
\end{proof}
\subsection{Graph matrices and $\nrank$}\label{sect-5.2}

In this section, we specialize the result of Lemma \ref{lemma-ncert} to the case of graph matrices $M_G$. As a result, we obtain a sufficient condition for lower bounding the $\nrank$.

We start with the following result that gives us information about the kernel of the matrices in the $\qrcert$s of the matrices $M_G$.
\begin{lemma}\label{main-lemma} 
Let $G=(V=[n], E)$ be a graph and let $r\geq 0$. Assume $\nrank(G)\leq r$, i.e., $M_G\in \CQ_n^{(r)}$. Let $P =  (P_{\beta})_{\beta \in \mathbb{N}_{r}^n}$  be a $\qrcert$ for $M_G$. Let $\beta \in \mathbb{N}_{r}^n$ and  let $C_1, C_2, \dots C_n$ be the columns $P_\beta$. Assume $S:=\supp(\beta)$ is stable in $G$ and $\alpha(G\setminus S^\perp)=\alpha(G)-|S|$. Then, for any critical edge $\{i,j\}$ of $G\setminus S^\perp$, we have $C_i=C_j$.
\end{lemma}
\begin{proof}
Since $\{i,j\}$ is critical in $G\setminus S^\perp$, then there exists $I\subseteq V$ such that $I\cup\{i\}$ and $I\cup\{j\}$ are stable of size $\alpha(G\setminus S^\perp)=\alpha(G)-|S|$ in $G\setminus S^\perp$. Then, $S\cup I \cup \{i\}$ and $S\cup I \cup\{j\}$ are stable of size $\alpha(G)$ in $G$.  Let $a=\chi^{S\cup I \cup \{i\}}$ and $b=\chi^{S\cup I \cup \{j\}}$ be the indicator vectors of  $S\cup I \cup \{i\}$ and $S\cup I \cup\{j\}$, respectively. Then, $a^TM_Ga=0$ and $b^TM_Gb=0$, in view of Theorem \ref{minimizers-LV}. Then, by Lemma \ref{lemma-ncert}, $a,b\in \ker P_\beta$, so that $a-b=\chi^{\{i\}}-\chi^{\{j\}}\in \ker P_\beta$. This implies $C_i=C_j$.
\end{proof}

Now, we state the main result of this section. For this, we introduce the graph
$$G_c=(V,E_c),$$
where $E_c=\{e\in E: e \text{ is critical in } G\}$ is the set of critical edges of $G$.
\begin{theorem}\label{theo-nrank}
Let $G=(V=[n],E)$ be a graph and let $S$ be a stable set of $G$ such that the following conditions hold:
\begin{description}
\item[(i)]$\alpha(G\setminus S^\perp)=\alpha(G)-|S|$.
\item[(ii)] For every subset $S'\subseteq S$ with $|S'|=|S|-2$ we have that the graph $(G\setminus S'^\perp)_c$ is connected.
 \end{description}
Then, we have $\nrank(G)\geq |S|-1$.
\end{theorem}
\begin{proof}
We show that $M_G \notin \CQ_n^{|S|-2}$ by contradiction. We set $|S|-2=r$.  Assume $M_G\in \CQ_n^{(r)}$, and let $P =  (P_{\beta})_{\beta \in \mathbb{N}_{r}^n}$  be a $\qrcert$ for $M_G$. Then, there exist scalars $c_A\geq 0$ (for $A\subseteq [n]$, with $|A|=r+2$)
such that the following equation holds:
\begin{align}\label{eq-Mg}
\Big(\sum_{i=1}^{n}x_i\Big)^{r} x^TM_Gx = \sum_{\beta\in \mathbb{N}_r^n} x^\beta x^TP_\beta x + \sum_{\substack{S\subseteq [n]\\ |A|=r+2}} c_Ax^A.
\end{align}
We will reach a contradiction by comparing the coefficient of $x^S$ in at both sides of Equation~(\ref{eq-Mg}). On the left-hand side, the coefficient is $-(r+2)(r+1)<0$. On the right-hand side, the coefficient of $x^S$ is
 $$\sum_{\substack{S'\subseteq V \\ S'\cup\{i,j\}=S}} 2(P_{S'})_{ij} + c_S.$$
We will show that all terms in the first summation are nonnegative.  Let $S'\subseteq S$, with $S'\cup\{i,j\}=S$. Observe that $\alpha(G\setminus S'^\perp)=\alpha(G)-|S'|$, because $\alpha(G\setminus S^\perp)=\alpha(G)-|S|$ and $S' \subseteq S$. By Lemma \ref{main-lemma}, if $\{v_1,v_2\}$ is a critical edge of $G\setminus S'^\perp$, then the columns of $P_{S'}$ indexed by $v_1$ and $v_2$ are equal. Using that $(G\setminus S'^\perp)_c$ is connected, we obtain that all columns of $P_{S'}$ indexed by vertices of $G\setminus S'^\perp$ are identical. In particular, the columns indexed by $i$ and $j$ are equal. This implies that $(P_{S'})_{ij}=(P_{S'})_{ii}$, which is nonnegative as $P_{S'}\succeq 0$. Using that $c_S\geq 0$, we reach a contradiction as the coefficient of $x^S$ on the right-hand side is positive while on the left-hand side it is negative.
\end{proof}

\section{Examples: Separating $\nnrank$ and $\nrank$}\label{sec-separation}
In this section, we consider two classes of graphs introduced, respectively, by  Dobre and Vera in \cite{Dobre_Vera_2015}, and by Vargas in \cite{LFV-thesis}. We show that all graphs in these classes satisfy the following two properties: 1)  $\nrank(G)$ is high and 2) $\nnrank(G)=1$.

\subsection*{Graphs $G_k$}
Dobre and Vera \cite{Dobre_Vera_2015} defined the following class of graphs $G_k$ for $k\geq 1$.
\begin{definition}
Let $K_{k+1,k+1}$ be the complete bipartite graph with bipartition $(U,V)$ with $U=\{u_0, u_1, \dots, u_k\}$ and $V=\{v_0,v_1, \dots, v_k\}$. The graph $G_{k}$ is obtained by adding a node $w_i$ in-between the edge $\{u_i,v_i\}$ for all $i=1,2,\dots, k$; that, is deleting the edge $\{u_i,v_i\}$ and adding the edges $\{u_i,w_i\}$ and $\{v_i,w_i\}$. We show $G_2$ and $G_k$ in Figure \ref{graphs-gk}.
\end{definition}

It was observed in \cite{Dobre_Vera_2015} that $\alpha(G_k)=k+1$, and it was conjectured that $\nrank(G_k)>k-1$. We show that $\nrank(G_k)>k-2$ as an application of Theorem \ref{theo-nrank}. Additionally, we show that $\nnrank(G_k)=1$ for all $k\geq 1$.
\begin{theorem}
Let $k\geq 1$ be an integer. Then, we have
\begin{description}
\item [(i)]  $\nrank(G_k)\geq  k-1$, i.e., $\nu^{(k-2)}(G_k)>\alpha(G_k)$.
\item [(ii)] $\nnrank(G_k)=1$.
\end{description}
\end{theorem}
\begin{proof}
For part (i) we use Theorem \ref{theo-nrank}. We set $S=\{w_1, w_2, \dots, w_k\}$. The first condition is satisfied, as we have $\alpha(G_k\setminus S^\perp)=1 = \alpha(G_k)-|S|$. Now, we check the second condition. Note that, for any subset $S'\subseteq S$ with $|S'|=k-2$, the graph $G\setminus S'^\perp$ is isomorphic to $G_2$. We observe that the edges $\{u_0,v_0\}$, $\{u_1, w_1\}$, $\{w_1, v_1\}$, $\{u_2, w_2\}$, $\{w_2, v_2\}$, $\{u_0, v_2\}$ and $\{v_2,u_0\}$ are critical in $G_2$. Therefore, $(G_2)_c$ is connected. Hence, by Theorem~\ref{theo-nrank}, we obtain that $\nrank(G_k)\geq k-1$.

To show part (ii), we use Proposition \ref{prop-lrank1}. We set $S= \{u_0, u_1, \dots, u_k, v_0, v_1, v_2, \dots, v_k\}$. To check condition (i), we observe that, for all $i\in S$, we have that the graph $G_k\setminus i^\perp$ is acyclic, and therefore it is perfect. Hence, $\nnrank(G\setminus i^\perp)=0$ for all $i\in S$. It is straightforward to check condition (ii), completing the proof.
\end{proof}

\definecolor{ududff}{rgb}{0.30196078431372547,0.30196078431372547,1}
\definecolor{ududff}{rgb}{0.30196078431372547,0.30196078431372547,1}
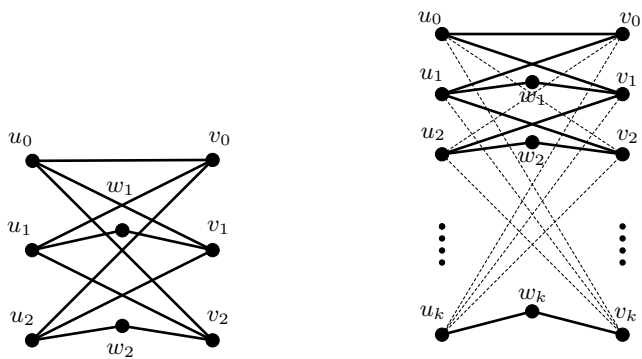
\begin{figure}[h]
\begin{tikzpicture}[line cap=round,line join=round,=triangle 45,x=.6cm,y=.6cm]
\clip(-9.26267376859976,1.4014566326578854) rectangle (-1.591894779446695,6.7389188190268845);
\draw [line width=1pt] (-7.99,5.98)-- (-4,6);
\draw [line width=1pt] (-7.99,5.98)-- (-4,2);
\draw [line width=1pt] (-8,2)-- (-4,6);
\draw [line width=1pt] (-8,4)-- (-6,4.44);
\draw [line width=1pt] (-6,4.44)-- (-4,4);
\draw [line width=1pt] (-8,2)-- (-6,2.32);
\draw [line width=1pt] (-6,2.32)-- (-4,2);
\draw [line width=1pt] (-8,4)-- (-4,2);
\draw [line width=1pt] (-8,2)-- (-4,4);
\draw [line width=1pt] (-8,-3)-- (-4,-3);
\draw [line width=1pt] (-8,-4)-- (-6,-3.8);
\draw [line width=1pt] (-6,-3.8)-- (-4,-4);
\draw [line width=1pt] (-8,-5)-- (-6,-4.8);
\draw [line width=1pt] (-6,-4.8)-- (-4,-5);
\draw [line width=1pt] (-8,-8)-- (-6.003143181611197,-7.612423467089333);
\draw [line width=1pt] (-6.003143181611197,-7.612423467089333)-- (-4,-8);
\draw [line width=1pt] (-8,-3)-- (-4,-4);
\draw [line width=1pt] (-8,-4)-- (-4,-3);
\draw [line width=1pt] (-8,-5)-- (-4,-4);
\draw [line width=1pt] (-8,-4)-- (-4,-5);
\draw [line width=1pt] (-7.99,5.98) --(-4,4) ;
\draw [line width=1pt] (-4,6)--(-8,4) ;

\draw [line width=0.4pt,dash pattern=on 1pt off 1pt] (-8,-3)-- (-4,-8);
\draw [line width=0.8pt,dash pattern=on 1pt off 1pt] (-4,-3)-- (-8,-8);
\draw [line width=0.8pt,dash pattern=on 1pt off 1pt] (-4,-4)-- (-8,-8);
\draw [line width=0.8pt,dash pattern=on 1pt off 1pt] (-8,-5)-- (-4,-8);
\draw [line width=0.8pt,dash pattern=on 1pt off 1pt] (-8,-4)-- (-4,-8);
%\draw [line width=0.8pt,dash pattern=on 1pt off 1pt] (-8,-5)-- (-4,-10);

\begin{scriptsize}
\draw [fill=black] (-7.99,5.98) circle (2.5pt);
\draw[color=black] (-8.24913928239037,6.451943810826468) node {$u_0$};
\draw [fill=black] (-8,4) circle (2.5pt);
\draw[color=black] (-8.271014127416473,4.475916143468546) node {$u_1$};
\draw [fill=black] (-8,2) circle (2.5pt);
\draw[color=black] (-8.212681207346867,2.4925968611019244) node {$u_2$};
\draw [fill=black] (-4,6) circle (2.5pt);
\draw[color=black] (-3.837712202126489,6.473818655852571) node {$v_0$};
\draw [fill=black] (-4,4) circle (2.5pt);
\draw[color=black] (-3.866878662161292,4.497790988494649) node {$v_1$};
\draw [fill=black] (-4,2) circle (2.5pt);
\draw[color=black] (-3.874170277169992,2.507180091119326) node {$v_2$};
\draw [fill=black] (-6,4.44) circle (2.5pt);
\draw[color=black] (-6.039779934754081,5.371745964060203) node {$w_1$};
\draw [fill=black] (-6,2.32) circle (2.5pt);
\draw[color=black] (-6.032488319745379,1.771760841554367) node {$w_2$};
\draw [fill=ududff] (6,12) circle (2.5pt);
\draw[color=ududff] (-3.3783404565783495,6.70610655148773);
\draw [fill=ududff] (6,10) circle (2.5pt);
\draw[color=ududff] (-3.028342936160719,6.70610655148773) ;
\draw [fill=ududff] (6,6) circle (2.5pt);
\draw[color=ududff] (-2.91896871103021,6.684231706461628);
\draw [fill=ududff] (4.325352867337781,-0.43918184673464467) circle (2.5pt);
\draw [fill=black] (-8,-3) circle (1pt);
\draw [fill=black] (-8,-5) circle (1pt);
\draw [fill=black] (-4,-3) circle (1pt);
\draw [fill=black] (-4,-5) circle (1pt);
\draw [fill=black] (-8,-8) circle (1pt);
\draw [fill=black] (-4,-8) circle (1pt);
\draw [fill=black] (-8,-4) circle (1pt);
\draw [fill=black] (-4,-4) circle (1pt);
\draw [fill=black] (-6,-3.8) circle (1pt);
\draw [fill=black] (-6,-4.8) circle (1pt);
\draw [fill=black] (-6.003143181611197,-7.612423467089333) circle (1pt);
\draw [fill=black] (-8,-6.2) circle (1pt);
\draw [fill=black] (-8,-6.4) circle (1pt);
\draw [fill=black] (-8,-6.8) circle (1pt);
\draw [fill=black] (-8,-6.6) circle (1pt);
\draw [fill=black] (-4,-6.2) circle (1pt);
\draw [fill=black] (-8,-6.2) circle (1pt);
\draw [fill=black] (-4,-6.4) circle (1pt);
\draw [fill=black] (-4,-6.6) circle (1pt);
\draw [fill=black] (-4,-6.8) circle (1pt);
\end{scriptsize}
\end{tikzpicture}
\begin{tikzpicture}[line cap=round,line join=round,=triangle 45,x=.6cm,y=.8cm]
\clip(-10.44623339079699,-8.543239234550395) rectangle (-1.6963974509522357,-2.4549503486506508);
\draw [line width=1pt] (-7.99,5.98)-- (-4,6);
\draw [line width=1pt] (-7.99,5.98)-- (-4,2);
\draw [line width=1pt] (-8,2)-- (-4,6);
\draw [line width=1pt] (-8,4)-- (-6,4.44);
\draw [line width=1pt] (-6,4.44)-- (-4,4);
\draw [line width=1pt] (-8,2)-- (-6,2.32);
\draw [line width=1pt] (-6,2.32)-- (-4,2);
\draw [line width=1pt] (-8,4)-- (-4,2);
\draw [line width=1pt] (-8,2)-- (-4,4);
\draw [line width=1pt] (-8,-3)-- (-4,-3);
\draw [line width=1pt] (-8,-4)-- (-6,-3.8);
\draw [line width=1pt] (-6,-3.8)-- (-4,-4);
\draw [line width=1pt] (-8,-5)-- (-6,-4.8);
\draw [line width=1pt] (-6,-4.8)-- (-4,-5);
\draw [line width=1pt] (-8,-8)-- (-6.003143181611197,-7.612423467089333);
\draw [line width=1pt] (-6.003143181611197,-7.612423467089333)-- (-4,-8);
\draw [line width=1pt] (-8,-3)-- (-4,-4);
\draw [line width=1pt] (-8,-4)-- (-4,-3);
\draw [line width=1pt] (-8,-5)-- (-4,-4);
\draw [line width=1pt] (-8,-4)-- (-4,-5);
\draw [line width=0.4pt,dash pattern=on 1pt off 1pt] (-8,-3)-- (-4,-8);
\draw [line width=0.4pt,dash pattern=on 1pt off 1pt] (-4,-3)-- (-8,-8);
\draw [line width=0.4pt,dash pattern=on 1pt off 1pt] (-4,-4)-- (-8,-8);
\draw [line width=0.4pt,dash pattern=on 1pt off 1pt] (-8,-5)-- (-4,-8);
\draw [line width=0.4pt,dash pattern=on 1pt off 1pt] (-8,-4)-- (-4,-8);
\draw [line width=0.4pt,dash pattern=on 1pt off 1pt] (-8,-8)-- (-4,-5);
\draw [line width=0.4pt,dash pattern=on 1pt off 1pt] (-8,-5)-- (-4,-3);
\draw [line width=0.4pt,dash pattern=on 1pt off 1pt] (-4,-5)-- (-8,-3);
\begin{scriptsize}
\draw [fill=black] (-7.99,5.98) circle (1pt);
\draw[color=black] (-6.994539555211844,-2.492378354096756);
\draw [fill=black] (-8,4) circle (1pt);
\draw[color=black] (-7.152568911539839,-2.492378354096756);
\draw [fill=black] (-8,2) circle (2.5pt);
\draw[color=black] (-7.260694260606361,-2.492378354096756);
\draw [fill=black] (-4,6) circle (2.5pt);
\draw[color=black] (-5.539006010085578,-2.492378354096756);
\draw [fill=black] (-4,4) circle (2.5pt);
\draw[color=black] (-5.547323344629157,-2.492378354096756);
\draw [fill=black] (-4,2) circle (2.5pt);
\draw[color=black] (-5.522371340998421,-2.492378354096756);
\draw [fill=black] (-6,4.44) circle (2.5pt);
\draw[color=black] (-6.2958834535512365,-2.492378354096756);
\draw [fill=black] (-6,2.32) circle (2.5pt);
\draw[color=black] (-6.3624221298998656,-2.492378354096756);
\draw [fill=ududff] (6,12) circle (2.5pt);
\draw[color=ududff] (-4.025251123154261,-2.492378354096756);
\draw [fill=ududff] (6,10) circle (2.5pt);
\draw[color=ududff] (-3.792365755934059,-2.492378354096756);
\draw [fill=ududff] (6,6) circle (2.5pt);
\draw[color=ududff] (-3.2351043415142886,-2.492378354096756);
\draw [fill=ududff] (4.325352867337781,-0.43918184673464467) circle (2.5pt);
\draw [fill=black] (-8,-3) circle (2.5pt);
\draw[color=black] (-8.258774405835801,-2.7335810558605433) node {$u_0$};
\draw [fill=black] (-8,-5) circle (2.5pt);
\draw[color=black] (-8.192235729487173,-4.663010684493791) node {$u_2$};
\draw [fill=black] (-4,-3) circle (2.5pt);
\draw[color=black] (-3.817317759564795,-2.7518023976655957) node {$v_0$};
\draw [fill=black] (-4,-5) circle (2.5pt);
\draw[color=black] (-3.8838564359134247,-4.703010684493791) node {$v_2$};
\draw [fill=black] (-8,-8) circle (2.5pt);
\draw[color=black] (-8.200553064030752,-7.632299116551453) node {$u_k$};
\draw [fill=black] (-4,-8) circle (2.5pt);
\draw[color=black] (-3.917125774087739,-7.6572511201821905) node {$v_k$};
\draw [fill=black] (-8,-4) circle (2.5pt);
\draw[color=black] (-8.225505067661487,-3.6815652083514832) node {$u_1$};
\draw [fill=black] (-4,-4) circle (2.5pt);
\draw[color=black] (-3.892173770457003,-3.7566132047207463) node {$v_1$};
\draw [fill=black] (-6,-3.8) circle (2.5pt);
\draw[color=black] (-5.979824740895247,-4.0653145102184323) node {$w_1$};
\draw [fill=black] (-6,-4.8) circle (2.5pt);
\draw[color=black] (-6.013094079069562,-5.063394655447898) node {$w_2$};
\draw [fill=black] (-6.003143181611197,-7.612423467089333) circle (2.5pt);
\draw[color=black] (-5.988142075438826,-7.349701727546825) node {$w_k$};
\draw [fill=black] (-8,-6.2) circle (1pt);
\draw [fill=black] (-8,-6.4) circle (1pt);
\draw [fill=black] (-8,-6.8) circle (1pt);
\draw [fill=black] (-8,-6.6) circle (1pt);
\draw [fill=black] (-4,-6.2) circle (1pt);
\draw [fill=black] (-8,-6.2) circle (1pt);
\draw [fill=black] (-4,-6.4) circle (1pt);
\draw [fill=black] (-4,-6.6) circle (1pt);
\draw [fill=black] (-4,-6.8) circle (1pt);
\end{scriptsize}
\end{tikzpicture}
\caption{Graphs $G_2$ and $G_k$}\label{graphs-gk}
\end{figure}
\subsection*{New class of graphs $L_k$}

We define the following class of graphs $L_k$. We start with a set of vertices $S_k=\{s_1, \dots, s_k\}$. For any pair of distinct nodes $s_i, s_j$ of $S$ we construct three extra nodes $a_{ij}$, $b_{ij}$, $c_{ij}$ and we construct the edges $\{s_i, a_{ij}\}$, $\{a_{ij}, b_{ij}\}$, $\{b_{ij}, s_j\}$, $\{s_j, c_{ij}\}$, $\{c_{ij}, s_i\}$ so that the nodes $s_i$, $a_{ij}$, $b_{ij}$, $s_j$, $c_{ij}$ form a 5-cycle. Finally, we construct a bipartite graph $K_{3,3}$ between the nodes $\{a_{ij}, b_{ij}, c_{ij}\}$ and $\{a_{lm}, b_{lm}, c_{lm}\}$ if $\{i,j\} \neq \{l,m\}$.
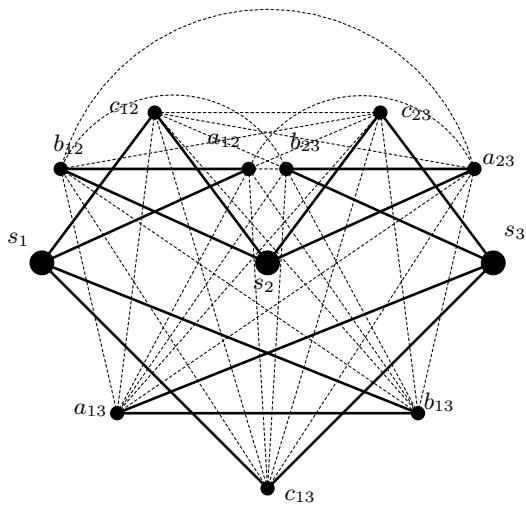
\begin{figure}[h!]
\begin{tikzpicture}[line cap=round,line join=round,=triangle 45,x=.5cm,y=.5cm]
\clip(-1.620697595216579,-8.078408518511932) rectangle (15.442714442414452,7.450199261766094);
\draw [line width=1pt] (5.5,2.5)-- (0,0);
\draw [line width=1pt] (6,0)-- (0.5,2.5);
\draw [line width=1pt] (0,0)-- (3,4);
\draw [line width=1pt] (6,0)-- (3,4);
\draw [line width=1pt] (0.5,2.5)-- (5.5,2.5);
\draw [line width=1pt] (6,0)-- (11.5,2.5);
\draw [line width=1pt] (11.5,2.5)-- (6.5,2.5);
\draw [line width=1pt] (6.5,2.5)-- (12,0);
\draw [line width=1pt] (9,4)-- (6,0);
\draw [line width=1pt] (9,4)-- (12,0);
\draw [line width=0.4pt,dash pattern=on 1pt off 1pt] (9,4)-- (3,4);
\draw [line width=1pt] (0,0)-- (10,-4);
\draw [line width=1pt] (12,0)-- (2,-4);
\draw [line width=1pt] (2,-4)-- (10,-4);
\draw [line width=1pt] (12,0)-- (6,-6);
\draw [line width=1pt] (6,-6)-- (0,0);
\draw [shift={(6,1.0637308395099503)},line width=0.4pt,dash pattern=on 1pt off 1pt]  plot[domain=0.25543543453778483:2.8861572190520084,variable=\t]({1*5.684440966478128*cos(\t r)+0*5.684440966478128*sin(\t r)},{0*5.684440966478128*cos(\t r)+1*5.684440966478128*sin(\t r)});
\draw [line width=0.4pt,dash pattern=on 1pt off 1pt] (5.5,2.5)-- (6.5,2.5);
\draw [line width=0.4pt,dash pattern=on 1pt off 1pt] (3,4)-- (2,-4);
\draw [line width=0.4pt,dash pattern=on 1pt off 1pt] (9,4)-- (10,-4);
\draw [line width=0.4pt,dash pattern=on 1pt off 1pt] (6,-6)-- (3,4);
\draw [line width=0.4pt,dash pattern=on 1pt off 1pt] (9,4)-- (6,-6);
\draw [line width=0.4pt,dash pattern=on 1pt off 1pt] (2,-4)-- (9,4);
\draw [line width=0.4pt,dash pattern=on 1pt off 1pt] (3,4)-- (10,-4);
\draw [line width=0.4pt,dash pattern=on 1pt off 1pt] (2,-4)-- (11.5,2.5);
\draw [line width=0.4pt,dash pattern=on 1pt off 1pt] (0.5,2.5)-- (2,-4);
\draw [line width=0.4pt,dash pattern=on 1pt off 1pt] (6,-6)-- (0.5,2.5);
\draw [line width=0.4pt,dash pattern=on 1pt off 1pt] (6,-6)-- (5.5,2.5);
\draw [line width=0.4pt,dash pattern=on 1pt off 1pt] (6,-6)-- (6.5,2.5);
\draw [line width=0.4pt,dash pattern=on 1pt off 1pt] (6,-6)-- (11.5,2.5);
\draw [line width=0.4pt,dash pattern=on 1pt off 1pt] (3,4)-- (6.5,2.5);
\draw [line width=0.4pt,dash pattern=on 1pt off 1pt] (9,4)-- (5.5,2.5);
\draw [line width=0.4pt,dash pattern=on 1pt off 1pt] (0.5,2.5)-- (9,4);
\draw [line width=0.4pt,dash pattern=on 1pt off 1pt] (3,4)-- (11.5,2.5);
\draw [line width=0.4pt,dash pattern=on 1pt off 1pt] (5.5,2.5)-- (2,-4);
\draw [line width=0.4pt,dash pattern=on 1pt off 1pt] (5.5,2.5)-- (10,-4);
\draw [line width=0.4pt,dash pattern=on 1pt off 1pt] (6.5,2.5)-- (2,-4);
\draw [line width=0.4pt,dash pattern=on 1pt off 1pt] (6.5,2.5)-- (10,-4);
\draw [shift={(3.5,1.1980513492228042)},line width=0.4pt,dash pattern=on 1pt off 1pt]  plot[domain=0.40945456110640277:2.7321380924833907,variable=\t]({1*3.270331831674052*cos(\t r)+0*3.270331831674052*sin(\t r)},{0*3.270331831674052*cos(\t r)+1*3.270331831674052*sin(\t r)});
\draw [shift={(8.5,1.1701532781760429)},line width=0.4pt,dash pattern=on 1pt off 1pt]  plot[domain=0.41725341098372776:2.7243392426060655,variable=\t]({1*3.281538100273395*cos(\t r)+0*3.281538100273395*sin(\t r)},{0*3.281538100273395*cos(\t r)+1*3.281538100273395*sin(\t r)});
\draw [line width=0.4pt,dash pattern=on 1pt off 1pt] (0.5,2.5)-- (10,-4);
\draw [line width=0.4pt,dash pattern=on 1pt off 1pt] (11.5,2.5)-- (10,-4);
\begin{scriptsize}
\draw [fill=black] (0,0) circle (4.5pt);
\draw[color=black] (-0.6275889581057782,0.6338627070510302) node {$s_1$};
\draw [fill=black] (6,0) circle (4.5pt);
\draw[color=black] (5.902852685925851,-0.5999995390563323) node {$s_2$};
\draw [fill=black] (12,0) circle (4.5pt);
\draw[color=black] (12.583765335580328,0.8144279137984491) node {$s_3$};
\draw [fill=black] (5.5,2.5) circle (2.5pt);
\draw[color=black] (4.8495556465659115,3.252058204888604) node {$a_{12}$};
\draw [fill=black] (0.5,2.5) circle (2.5pt);
\draw[color=black] (0.7266500924998592,3.161775601514895) node {$b_{12}$};
\draw [fill=black] (3,4) circle (2.5pt);
\draw[color=black] (2.2012659476037757,4.124790037501128) node {$c_{12}$};
\draw [fill=black] (6.5,2.5) circle (2.5pt);
\draw[color=black] (7.016338127534931,3.2219640037640347) node {$b_{23}$};
\draw [fill=black] (11.5,2.5) circle (2.5pt);
\draw[color=black] (12.162446519836353,2.7404567857709172) node {$a_{23}$};
\draw [fill=black] (9,4) circle (2.5pt);
\draw[color=black] (9.965569837742763,3.97431903187828) node {$c_{23}$};
\draw [fill=black] (2,-4) circle (2.5pt);
\draw[color=black] (1.298439913866684,-3.910361662759012) node {$a_{13}$};
\draw [fill=black] (10,-4) circle (2.5pt);
\draw[color=black] (10.567453860234158,-3.6997022548870238) node {$b_{13}$};
\draw [fill=black] (6,-6) circle (2.5pt);
\draw[color=black] (6.865867121912083,-6.197520948226319) node {$c_{13}$};
\end{scriptsize}
\end{tikzpicture}
\caption{Graph $L_3$}
\end{figure}

\begin{lemma}
Let $L_k$ be as above, then $\alpha(L_k)=k$.
\end{lemma}
\begin{proof}
Note that $S_k$ is stable of size $k$ in $L_k$. We now show that there is no stable set of size $k+1$ in $L_k$. Let $A\subseteq V$ be a stable set in $L_k$. By construction, $A$ could contain elements of type $a_{ij}$ $b_{ij}$ and $c_{ij}$ from just one pair $(i,j)$. Assume that $(1,2)$ is such pair, so that $\{a_{ij}, b_{ij}, c_{ij}\} \cap A=\emptyset$ if $(i,j)\neq (1,2)$. Hence, $A$ is stable in the graph obtained by deleting all nodes $a_{ij}, b_{ij}$ and $c_{ij}$ for $(i,j)\neq (1,2)$, which is isomorphic to the graph $C_5 \oplus \overline{K_{k-2}}$. Hence, $|A|\leq \alpha(C_5\oplus \overline{K_{k-2}}) = k$.
\end{proof}

We use the following result.
\begin{lemma}\label{lemma-sub}
Let $H, G$ be two graphs with $V(H)=V(G)$, $E(H)\subseteq E(G)$, and $\alpha(H)=\alpha(G)$. Then, $\nnrank(G)\leq \nnrank(H)$.
\end{lemma}
\begin{proof}
Observe that $M_G\geq M_H$, therefore $(\sum_{\val}c_{\val}x^{\val})x^TM_Gx\geq_c (\sum_{\val}c_{\val}x^{\val})x^TM_Hx$.
\end{proof}
\begin{theorem}\label{coro-L_k}
For any $k\geq 2$, we have
\begin{itemize}
 \item [(i)] $\nrank(L_k)\geq \alpha(L_k)-1$, i.e., $\nu^{(\alpha(L_k)-2)}(L_k)>\alpha(L_k)$.
 \item [(ii)] $\nnrank(L_k)=1$.
 \end{itemize}
\end{theorem}
\begin{proof}
For part (i) we apply Theorem \ref{theo-nrank}. We check that the set $S_k$ satisfies the conditions. Note that $G\setminus S_k^\perp$ is the empty graph, so we have $0=\alpha(L_k\setminus S_k^\perp)= \alpha(L_k)-|S_k|$. Now, for any subset $S'\subseteq S_k$ with $|S'|=k-2$, the graph $L_k\setminus S'^\perp$ is isomorphic to $C_5$, which is critical and connected. Then, by Theorem \ref{theo-nrank}, we have  $\nrank(L_k)\geq |S_k|-1=\alpha(L_k)-1$.

For part (ii), we define the graphs $L_k'$, obtained by starting with the graph $L_k$ and deleting all edges of the form $\{a_{ij},c_{mn}\}$ and $\{b_{ij}, c_{mn}\}$ when $\{i,j\}\neq \{m,n\}$. Observe that $C_1=\{c_{ij}: \{i,j\}\in \{1, \dots, k\}^2 \}$ and $C_2=\{ a_{ij}, b_{ij}:  \{i,j\} \in \{1, \dots, k\}^2\}$ induce cliques in  $L_k$. We now show that $\nnrank(L_k')=1$ by using Proposition \ref{prop-lrank1}. For this, we set $S=C_2$. Using that $C_2$ is a clique, we observe that, for every $i\in C_2$, the graph $L_k'\setminus i^\perp$ has no cycles and therefore $\nnrank(L_k'\setminus i^\perp)=0$. Condition (ii) is straightforward to check.

Note that $V(L_k')=V(L_k)$, and $E(L_k')\subseteq E(L_k)$. Then, in order to apply Lemma \ref{lemma-sub}, it remains to prove that $\alpha(L_k')=k$. First, observe that $\alpha(L_k)\geq k$, as $S_k$ is stable. Let $A$ be a stable set in $L_k'$. We will show that $|A|\leq k$. First, assume there exists $i\in A\cap C_2$. Then $A\setminus \{i\}$ is stable in $L_k'\setminus i^\perp$. It is easy to observe that, for $i\in C_2$, we have $\alpha(L_k'\setminus i^\perp)\leq k-2$. Thus, in this case, $|A|\leq k-1$. Now, we assume $A\cap C_2 =\emptyset$. Then, $A$ is stable in the induced subgraph $L_k'[S_k\cup C_1]$. It is easy to observe that $\alpha(L_k'[S_k\cup C_1])=k$, thus showing that $|A|\leq k$.
\end{proof}

\section{New polyhedral-based hierarchy}
In this section, we consider a new inner linear approximation for $\COP_n$. In this paper, we have introduced the cones $\newQ_n^{(r)}$, which are based on a certificate of copositivity for a matrix $M\in\CS^n$ using sums of squares of polynomials. Now, we will examine a different and weaker certificate for copositivity. Specifically, suppose there exists a nonzero polynomial $p \in \mathbb{R}[x_1, \dots, x_n]_r$ with nonnegative coefficients such that the expression $p \cdot x^T M x$ also has nonnegative coefficients. In that case, we can conclude that $M$ is copositive. We will consider the following cones, denoted as $\tilde{\mathcal{C}}_n^{(r)} \subseteq \COP_n$.
\[
\tilde{\mathcal{C}}_n^{(r)}=\Big\{M\in \CS^n: p\cdot x^TMx \in \Rplus{n,r+2} \text{ for some } p\in \Rplus{n,r}  \text{ with } \|p\|=1 \Big\}.
\]

Similar to the cones $\newQ_n^{(r)}$, the cones $\tilde{\mathcal{C}}_n^{(r)}$ serve as a generalization of an existing class of cones $\CC_n^{(r)}$ introduced by de Klerk and Pasechnik \cite{deKlerk_Pasechnik_2002}. The key distinction is that for the cones  $\tilde{\mathcal{C}}_n^{(r)}$ we allow the multiplier $p$ to be any nonzero degree-$r$ homogeneous polynomial with nonnegative coefficients, while in the definition of the cones $\CC_n^{(r)}$ this multiplier is always $(\sum_{i=1}^nx_i)^r$.

In \cite{deKlerk_Pasechnik_2002} de Klerk and Pasechnik introduced the bounds $\zeta^{(r)}(G)$ as the relaxation of problem~(\ref{alpha-cop}) obtained through the substitution of $\COP_n$ with the cones $\CC_n^{(r)}$. We now introduce the bounds $\tilde{\zeta}^{(r)}(G)$ as follows.
\begin{align}\label{zeta-tilde}
\tilde{\zeta}^{(r)}(G)=\min\Big\{t: t(A_G+I)-J\in \tilde{\mathcal{C}}_{n}^{(r)}\}
\end{align}

Notice that testing the membership of a matrix symmetric $M\in \CS^n$ in the cone $\tilde{\CC}_n^{(r)}$ amounts to check whether a system of linear inequalities is feasible. Then, for a fixed $r\in \mathbb{N}$, the parameter $\tilde{\zeta}^{(r)}(G)$ can be approximated (up to given precision) by performing a binary search in the interval $[0,n]$. Notice that this procedure can be performed in polynomial time.

Clearly, $\alpha(G)\leq \tilde{\zeta}^{(r)}(G)\leq  \zeta^{(r)}(G)$ for all $r\in \N$. It is known that $\zeta^{(r)}(G)$ converges asymptotically to $\alpha(G)$ as $r \to \infty$ \cite{deKlerk_Pasechnik_2002}. Thus, the same result holds for the hierarchy $\tilde{\zeta}^{(r)}(G)$.

As shown in \cite{deKlerk_Pasechnik_2002} (see also \cite{Pena_Vera_Zuluaga_2007}) the hierarchy $\zeta^{(r)}(G)$ does not have finite convergence to $\alpha(G)$ unless $G$ is a complete graph. A natural question to ask is whether finite convergence is achieved by the hierarchy $\tilde{\zeta}^{(r)}(G)$; we answer this question negatively next. 

\begin{theorem}\label{main-linear}
Let $G$ be a graph which is not complete. Then $\tilde{\zeta}^{(r)}(G)>\alpha(G)$ for all $r\in \N$.
\end{theorem}

Theorem~\ref{main-linear} is equivalent to showing that if
$p\in \mathcal{N}_{n,r}$ is such that
\begin{equation}\label{finite-linear}
 p\cdot x^TM_Gx \in \Rplus{n,r+2}
\end{equation} 
then $p=0$. To show the theorem, we first prove it for star graphs, as detailed in \cref{fan}. 
Subsequently, we show that if equation~\eqref{finite-linear} is satisfied for a graph \( G \), it follows that the same equation must hold for the star graph with \( \alpha(G) \) leaves. In pursuit of this objective, we must first introduce some preliminary results.

\begin{lemma}
Let $G$ be the graph consisting of $n$ isolated nodes. Let $r>0$ and assume that equation (\ref{finite-linear}) holds for some $p\in \mathcal{N}_{n,r}$. Then, $p=0$.
\end{lemma}
\begin{proof} Let $r>0$.
Assume that there exist some nonnegative scalars $c_\alpha$ (for $\alpha\in \mathbb{N}_r^{n}$) and $a_\beta$ (for $\beta \in \mathbb{N}_{r+2}^n$) such that the following identity holds.
\[
\left(\textstyle \sum_{\alpha\in \N^n_r}  c_\alpha x^{\alpha}\right) x^TM_Gx  = \textstyle\sum_{\beta\in \N^n_{r+2}}a_\beta x^\beta,
\]
 Now we replace $x=(1,1, \dots, 1)$. The left-hand side is equal to zero, as for such $x$ we have $x^TM_Gx = n \sum_i x_i^2 - (\sum_i x_i)^2 =0$. Then, we obtain $\sum a_\beta =0$, and thus $a_\beta=0$ for all $\beta\in \N^n_{r+2}$. This implies that $p=0$.
\end{proof}
\begin{lemma}\label{lemma-isolated}
Let $r >0$ and let $G$ be a graph such that equation (\ref{finite-linear}) holds with \linebreak $p= \sum_{\alpha\in \N^n_r}  c_\alpha x^{\alpha}\in \mathcal{N}_{n,r}$. Let $S$ be a stable set with $|S|=\alpha(G)$. Then, $c_\alpha =0$ for every $\alpha\in \mathbb{N}_r^{n}$ such that $supp(\alpha)\subseteq S$.
\end{lemma}
\begin{proof}
By replacing $x_i=0$ for $i\notin S$ we obtain an equation of the form (\ref{finite-linear}) for the graph consisting of $\alpha(G)$ isolated nodes. By the previous lemma, we obtain that $c_\alpha =0$ for every $\alpha\in \mathbb{N}_r^{n}$ such that $supp(\alpha)\subseteq S$.
\end{proof}
\begin{lemma}\label{fan} Let $T_n$ be the \textit{star graph},  with vertex set $\{0,1,\dots, n\}$ and edge set $\{\{0,i\}:i=1,\dots,n\}$. Then,
$\tilde{\zeta}^{(r)}(T_n) >n = \alpha(T_n)$ for every $r\in \N$.
\end{lemma}
\begin{proof}
Assume, by sake of contradiction, that $\tilde{\zeta}^{(r)}(T_n)=n$, and that $r$ is minimum. Notice that $r >0$ as $M_{T_n}$ has some negative entries. Let  $p= \sum_{\alpha\in \N^n_r}  c_\alpha x^{\alpha}\in \mathcal{N}_{n,r}$ be the multiplier in equation (\ref{finite-linear}). By Lemma \ref{lemma-isolated}, we have that $c_\alpha=0$ for all $\alpha\in \mathbb{N}_r^{n}$ such that $supp(\alpha)\subseteq \{1, 2, \dots, n\}$. This implies that $p$ is divisible by $x_0$. Since $p\neq 0$, then $\frac{1}{x_0}p\in \mathcal{N}_{n,r-1}$ is a polynomial with nonnegative coefficients and $(\frac{1}{x_0}p)\cdot x^TM_Gx \in \mathcal{N}_{n,r+1}$. Hence, $\tilde{\zeta}^{(r-1)}(T_n)=\alpha(T_n)$ contradicting the minimality of $r$.
\end{proof}
\begin{proof}[Proof of Theorem \ref{main-linear}]
For sake of contradiction, assume $r$ is such that $\tilde{\zeta}^{(r)}(G)=\alpha(G)$. Let $p= \sum_{\substack {\alpha\in \N^n_r}}  c_\alpha x^{\alpha} \in \mathcal{N}_{n,r}\setminus\{0\}.$
Let $S$ be a stable set in $G$ of size $\alpha(G)$. Let $H = (V(G),E)$ be the graph with edge set $E =\{\{u,v\}: u\neq v, \{u,v\} \nsubseteq S \}$ that is, $H$ is obtained by adding to $G$ all edges not connecting two vertices from $S$. In particular,   $\alpha(H) = \alpha(G)$ and
$M_H - M_G$ has non-negative entries. Thus
$p(x) (x^TM_Hx) = p(x) (x^TM_Gx) + p(x) (x^T(M_H - M_G)x) \in \Rplus{n,r+2}$.  Now, by replacing $x_v \leftarrow \tfrac {1}{|V\setminus S|}x_0$ for all $v \in V\setminus S$ we obtain that the polynomial $x^TM_{H}x$ is transformed into $x^TM_{T_{\alpha(G)}}x$, where $T_{\alpha(G)}$ is the graph as in Lemma \ref{fan}. Also, we observe that performing this substitution transforms $p$ into $p' \in \Rplus{1+\alpha(G),r+2}$ where $p'\neq 0$, as this transformation only adds groups of (nonnegative) coefficients. Therefore, $\tilde{\zeta}^{(r)}(T_{\alpha(G)})=\alpha(G)$, reaching a contradiction in view of Lemma~\ref{fan}.
\end{proof}

\section{Conclusion and Discussion}
The hierarchy $\newQ^{(r)}$ proposed in this paper is closed under borderings and diagonal scaling; however, it is not convex in general. It is natural to consider the existence of a hierarchy of convex cones that maintains closure under both scaling and borderings, and that approximates the copositive cone. Notably, the cones $\CQ^{(0)}$ and $\COP$ exhibit all three desired properties. Nonetheless, it remains an open question as to whether there exists even an intermediate cone that satisfies these criteria.

The equation \( \newQ^{(1)}_5 = \COP_5 \) establishes that the assessment of copositivity for \( 5 \times 5 \) matrices can be effectively reformulated as an SDP-feasibility problem. Notably,  \citet{deKP2007COP2LP} has shown that testing for copositivity can be formulated as an exponential-size LP-feasibility problem. Nonetheless, it remains an open question whether the case \( n=5 \) is exceptional for the \( \newQ^{(r)} \) hierarchy, or if there exists a convergence pattern across all dimensions; is there a function \( r(n) \) such that \( \newQ^{r(n)}_n = \COP_n \)? While a small \( r(n) \) is possible, it would imply significant consequences regarding the NP-hardness of the SDP feasibility problem.

The computational complexity of determining (or approximating up to fixed precision) $\tilde{\nu}^{(r)}(G)$ remains unclear. As discussed in Section \ref{sec:nuTilde}, approximating $\tilde{\nu}^{(r)}(G)$ involves solving $\log(n)$ semidefinite programming (SDP) feasibility problems. However, the complexity of general SDP feasibility is still unresolved. Therefore, we pose an open question:  can $\tilde{\nu}^{(r)}(G)$  be computed (up to a fixed precision) in polynomial time for fixed $r\in \mathbb{N}$?. In contrast, it has recently been shown that for fixed $r\in \mathbb{N}$, the bounds $\vartheta^{(r)}(G)$ and $\nu^{(r)}(G)$ can be computed (up to fixed precision) in polynomial time \cite{palomba}.

We have shown that the $\nnrank$ is low (in fact equal to one) for certain graph classes for which other SDP approaches require a high degree.  Notice that $\nnrank(G)$ cannot be bounded by a universal constant $c$ for all graphs $G$, unless SDP feasibility is NP-hard. So far, we have only produced graphs with $\nnrank$ at most 2. It remains an open problem to find a family of graphs with unbounded $\nnrank$. 

We have introduced a new linear hierarchy of approximations, denoted by $\tilde{\zeta}^{(r)}(G)$, for the independence number $\alpha(G)$. This hierarchy strengthens the previously studied $\zeta^{(r)}(G)$ hierarchy. It is known \cite{Pena_Vera_Zuluaga_2007} that $\zeta^{(r)}(G)$ depends only on $\alpha(G)$, on the size of the graph $|G|$, and on the level $r$. As a result, it is independent of the actual structure of the graph. Although this dependence simplifies the analysis of the hierarchy, it also limits the hierarchy's ability to capture meaningful graph-specific features. A natural and important question arises: Does the new hierarchy $\tilde{\zeta}^{(r)}(G)$ share this same limitation, or does it capture other structural properties of the graph?

\bibliographystyle{plainnat}
\bibliography{Qtilde}

\begin{thebibliography}{38}
\providecommand{\natexlab}[1]{#1}
\providecommand{\url}[1]{\texttt{#1}}
\expandafter\ifx\csname urlstyle\endcsname\relax
  \providecommand{\doi}[1]{doi: #1}\else
  \providecommand{\doi}{doi: \begingroup \urlstyle{rm}\Url}\fi

\bibitem[Au and Tun{\c{c}}el(2024)]{LeventExample}
Yu~Hin Au and Levent Tun{\c{c}}el.
\newblock Stable set polytopes with high lift-and-project ranks for the lov{\'a}sz--schrijver sdp operator.
\newblock \emph{Mathematical Programming}, pages 1--36, 2024.

\bibitem[Bodirsky et~al.(2024{\natexlab{a}})Bodirsky, Kummer, and Thom]{bodirsky2024spectrahedral}
Manuel Bodirsky, Mario Kummer, and Andreas Thom.
\newblock Spectrahedral shadows and completely positive maps on real closed fields.
\newblock \emph{Journal of the European Mathematical Society}, 2024{\natexlab{a}}.

\bibitem[Bodirsky et~al.(2024{\natexlab{b}})Bodirsky, Kummer, and Thom]{specShadows2024}
Manuel Bodirsky, Mario Kummer, and Andreas Thom.
\newblock Spectrahedral shadows and completely positive maps on real closed fields.
\newblock \emph{Journal of the European Mathematical Society}, 2024{\natexlab{b}}.
\newblock \doi{10.4171/JEMS/1509}.
\newblock Published online first.

\bibitem[Bomze(2012)]{bomze2012copositive}
Immanuel~M Bomze.
\newblock Copositive optimization--recent developments and applications.
\newblock \emph{European Journal of Operational Research}, 216\penalty0 (3):\penalty0 509--520, 2012.

\bibitem[Bomze and de~Klerk(2002)]{Bomze2002}
Immanuel~M. Bomze and Etienne de~Klerk.
\newblock Solving standard quadratic optimization problems via linear, semidefinite and copositive programming.
\newblock \emph{J. Glob. Optim.}, 24\penalty0 (2):\penalty0 163--185, 2002.
\newblock \doi{10.1023/A:1020209017701}.

\bibitem[Bomze et~al.(2000)Bomze, D{\"{u}}r, de~Klerk, Roos, Quist, and Terlaky]{Bomze2000}
Immanuel~M. Bomze, Mirjam D{\"{u}}r, Etienne de~Klerk, Cornelis Roos, Arie~J. Quist, and Tam{\'{a}}s Terlaky.
\newblock On copositive programming and standard quadratic optimization problems.
\newblock \emph{J. Glob. Optim.}, 18\penalty0 (4):\penalty0 301--320, 2000.
\newblock \doi{10.1023/A:1026583532263}.

\bibitem[Bomze et~al.(2017)Bomze, Schachinger, and Ullrich]{Bomze_Schachinger_Ullrich_2017}
Immanuel~M. Bomze, Werner Schachinger, and Reinhard Ullrich.
\newblock The complexity of simple models---a study of worst and typical hard cases for the standard quadratic optimization problem.
\newblock \emph{Mathematics of Operations Research}, 2017.

\bibitem[Burer(2015)]{Burer2015AGG}
Samuel Burer.
\newblock A gentle, geometric introduction to copositive optimization.
\newblock \emph{Mathematical Programming}, 151:\penalty0 89 -- 116, 2015.
\newblock URL \url{https://api.semanticscholar.org/CorpusID:254130757}.

\bibitem[de~Klerk and Pasechnik(2007)]{deKP2007COP2LP}
E.~de~Klerk and D.~V. Pasechnik.
\newblock A linear programming reformulation of the standard quadratic optimization problem.
\newblock \emph{Journal of Global Optimization}, 37\penalty0 (1):\penalty0 75--84, 2007.
\newblock ISSN 1573-2916.
\newblock \doi{10.1007/s10898-006-9037-9}.
\newblock URL \url{https://doi.org/10.1007/s10898-006-9037-9}.

\bibitem[de~Klerk and Pasechnik(2002)]{deKlerk_Pasechnik_2002}
Etienne de~Klerk and Dmitrii~V. Pasechnik.
\newblock Approximation of the stability number of a graph via copositive programming.
\newblock \emph{SIAM Journal on Optimization}, 12\penalty0 (4):\penalty0 875--892, 2002.

\bibitem[Dickinson(2013)]{Dickinson_2013}
Peter~J.C. Dickinson.
\newblock \emph{The Copositive Cone, the Completely Positive Cone and their Generalisations}.
\newblock PhD thesis, University of Groningen, 2013.

\bibitem[Dickinson and Povh(2019)]{Dickinson_Povh_2013}
Peter~J.C. Dickinson and Janez Povh.
\newblock A new tractable approximation hierarchy for general polynomial optimization problems.
\newblock \emph{Computational Optimization and Applications}, 73\penalty0 (1):\penalty0 37--67, January 2019.
\newblock ISSN 0926-6003.
\newblock \doi{10.1007/s10589-019-00066-0}.

\bibitem[Dickinson et~al.(2013)Dickinson, D\"ur, Gijben, and Hildebrand]{Dickinson_Duer_Gijben_Hildebrand_2013}
Peter~J.C. Dickinson, Mirjam D\"ur, Luuk Gijben, and Roland Hildebrand.
\newblock Scaling relationship between the copositive cone and parrilo's first level approximation.
\newblock \emph{Optimization Letters}, 7\penalty0 (8):\penalty0 1669--1679, 2013.

\bibitem[Dobre and Vera(2015)]{Dobre_Vera_2015}
Cristian Dobre and Juan Vera.
\newblock Exploiting symmetry in copositive programs via semidefinite hierarchies.
\newblock \emph{Mathematical Programming}, 151\penalty0 (2):\penalty0 659--680, 2015.

\bibitem[Dong(2013)]{Hongbo_13}
Hongbo Dong.
\newblock Symmetric tensor approximation hierarchies for the completely positive cone.
\newblock \emph{SIAM Journal on Optimization}, 23\penalty0 (3):\penalty0 1850--1866, 2013.
\newblock \doi{10.1137/100813816}.
\newblock URL \url{https://doi.org/10.1137/100813816}.

\bibitem[Dukanovic and Rendl(2010)]{Dukanovic2010}
Igor Dukanovic and Franz Rendl.
\newblock Copositive programming motivated bounds on the stability and the chromatic numbers.
\newblock \emph{Math. Program.}, 121\penalty0 (2):\penalty0 249--268, 2010.
\newblock \doi{10.1007/S10107-008-0233-X}.

\bibitem[D{\"u}r(2010)]{dur2010copositive}
Mirjam D{\"u}r.
\newblock Copositive programming--a survey.
\newblock In \emph{Recent Advances in Optimization and its Applications in Engineering: The 14th Belgian-French-German Conference on Optimization}, pages 3--20. Springer, 2010.

\bibitem[D{\"u}r and Rendl(2021)]{Dr2021ConicOA}
Mirjam D{\"u}r and Franz Rendl.
\newblock Conic optimization: A survey with special focus on copositive optimization and binary quadratic problems.
\newblock \emph{EURO J. Comput. Optim.}, 9:\penalty0 100021, 2021.
\newblock URL \url{https://api.semanticscholar.org/CorpusID:237257784}.

\bibitem[Gvozdenovi{\'c} and Laurent(2007)]{GL07}
N.~Gvozdenovi{\'c} and M.~Laurent.
\newblock Semidefinite bounds for the stability number of a graph via sums of squares of polynomials.
\newblock \emph{Mathematical Programming}, 110\penalty0 (1):\penalty0 145--173, 2007.
\newblock ISSN 0025-5610.

\bibitem[Hall and Newman(1963)]{Hall_Newman_1963}
Marshall Hall and Morris Newman.
\newblock Copositive and completely positive quadratic forms.
\newblock \emph{Mathematical Proceedings of the Cambridge Philosophical Society}, 59\penalty0 (2):\penalty0 329–339, 1963.
\newblock \doi{10.1017/S0305004100036951}.

\bibitem[Haynsworth and Hoffman(1969)]{Haynsworth_Hoffman_1969}
Emilie Haynsworth and Alan~J. Hoffman.
\newblock Two remarks on copositive matrices.
\newblock \emph{Linear Algebra and its Applications}, 2:\penalty0 387--392, 1969.

\bibitem[Hoffman and Pereira(1973)]{Hoffman_Pereira_1973}
Alan~J. Hoffman and Francisco Pereira.
\newblock On copositive matrices with -1,0,1 entries.
\newblock \emph{Journal of Combinatorial Theory}, 14\penalty0 (3):\penalty0 302--309, 1973.

\bibitem[Kostyukova and Tchemisova(2022)]{Kostyukova2022}
Olga~I. Kostyukova and Tatiana~V. Tchemisova.
\newblock On strong duality in linear copositive programming.
\newblock \emph{J. Glob. Optim.}, 83\penalty0 (3):\penalty0 457--480, 2022.
\newblock \doi{10.1007/S10898-021-00995-3}.

\bibitem[Kostyukova and Tchemisova(2023)]{Kostyukova2023}
Olga~I. Kostyukova and Tatiana~V. Tchemisova.
\newblock An exact explicit dual for the linear copositive programming problem.
\newblock \emph{Optim. Lett.}, 17\penalty0 (1):\penalty0 107--120, 2023.
\newblock \doi{10.1007/S11590-022-01870-0}.

\bibitem[Ko\v{c}vara and Stingl(2003)]{Kocvara_Stingl_2003}
Michal Ko\v{c}vara and Michael Stingl.
\newblock Pennon: A code for convex nonlinear and semidefinite programming.
\newblock \emph{Optimization Methods and Software}, 18\penalty0 (3):\penalty0 317--333, 2003.
\newblock ISSN 1055-6788.
\newblock \doi{10.1080/1055678031000098773}.

\bibitem[Lasserre(2001)]{lasserre2001_0_1}
Jean~B Lasserre.
\newblock An explicit exact sdp relaxation for nonlinear 0-1 programs.
\newblock In \emph{International Conference on Integer Programming and Combinatorial Optimization}, pages 293--303. Springer, 2001.

\bibitem[Laurent and Vargas(2021)]{LV21b}
Monique Laurent and Luis~Felipe Vargas.
\newblock Exactness of parrilo's conic approximations for copositive matrices and associated low order bounds for the stability number of a graph.
\newblock \emph{Math. Oper. Res.}, 48:\penalty0 1017--1043, 2021.
\newblock URL \url{https://api.semanticscholar.org/CorpusID:237940398}.

\bibitem[Laurent and Vargas(2022)]{LV21a}
Monique Laurent and Luis~Felipe Vargas.
\newblock Finite convergence of sum-of-squares hierarchies for the stability number of a graph.
\newblock \emph{SIAM Journal on Optimization}, 32\penalty0 (2):\penalty0 491--518, 2022.
\newblock \doi{10.1137/21M140345X}.
\newblock URL \url{https://doi.org/10.1137/21M140345X}.

\bibitem[Lov{\'a}sz and Schrijver(1991)]{lovaszSchrijver}
L{\'a}szl{\'o} Lov{\'a}sz and Alexander Schrijver.
\newblock Cones of matrices and set-functions and 0--1 optimization.
\newblock \emph{SIAM journal on optimization}, 1\penalty0 (2):\penalty0 166--190, 1991.

\bibitem[Murty and Kabadi(1985)]{murty1985coNP}
Katta~G Murty and Santosh~N Kabadi.
\newblock Some np-complete problems in quadratic and nonlinear programming.
\newblock Technical report, 1985.

\bibitem[Palomba et~al.(2025)Palomba, Slot, Vargas, and Mastrolilli]{palomba}
Marilena Palomba, Lucas Slot, Luis~Felipe Vargas, and Monaldo Mastrolilli.
\newblock {Computational complexity of sum-of-squares bounds for copositive programs}, 2025.
\newblock URL \url{https://arxiv.org/abs/2501.03698}.

\bibitem[Parrilo(2000)]{Parrilo_Thesis}
Pablo~A Parrilo.
\newblock \emph{Structured semidefinite programs and semialgebraic geometry methods in robustness and optimization}.
\newblock PhD thesis, California Institute of Technology, 2000.

\bibitem[Pe\~na et~al.(2007)Pe\~na, Vera, and Zuluaga]{Pena_Vera_Zuluaga_2007}
Javier~F. Pe\~na, Juan~C. Vera, and Luis~F. Zuluaga.
\newblock Computing the stability number of a graph via linear and semidefinite programming.
\newblock \emph{SIAM Journal on Optimization}, 18\penalty0 (1):\penalty0 87--105, 2007.

\bibitem[Pe\~na et~al.(2014)Pe\~na, Vera, and Zuluaga]{Pena_Vera_Zuluaga_2014}
Javier~F. Pe\~na, Juan~C. Vera, and Luis~F. Zuluaga.
\newblock Completely positive reformulations for polynomial optimization.
\newblock \emph{Mathematical Programming}, 151\penalty0 (2):\penalty0 405--431, 2014.

\bibitem[Schrijver(1979)]{S79}
A.~Schrijver.
\newblock A comparison of the delsarte and lov{\'a}sz bounds.
\newblock \emph{IEEE Transactions on Information Theory}, 25\penalty0 (4):\penalty0 425--429, 1979.

\bibitem[Schweighofer and Vargas(2024)]{SV24}
Markus Schweighofer and Luis~Felipe Vargas.
\newblock Sum-of-squares certificates for copositivity via test states.
\newblock \emph{{SIAM} J. Appl. Algebra Geom.}, 8\penalty0 (4):\penalty0 797--820, 2024.
\newblock \doi{10.1137/23M1611798}.
\newblock URL \url{https://doi.org/10.1137/23m1611798}.

\bibitem[Vargas(2023)]{LFV-thesis}
Luis Vargas.
\newblock \emph{Sum-of-squares representations for copositive matrices and independent sets in graphs}.
\newblock PhD thesis, Tilburg University, 2023.
\newblock CentER Dissertation Series Volume: 719.

\bibitem[Zykov(1949)]{zykov1952}
A.~A. Zykov.
\newblock On some properties of linear complexes.
\newblock \emph{Matematicheskii Sbornik (Mathematical Sbornik)}, 24\penalty0 (66):\penalty0 163--188, 1949.
\newblock URL \url{https://www.mathnet.ru/eng/sm5974}.

\end{thebibliography}

\end{document}